\documentclass[12pt,reqno,a4paper]{amsart}
\usepackage[utf8]{inputenc}
\usepackage[T1]{fontenc}
\usepackage{amssymb}
\usepackage{pgf,tikz,pgfplots}
\pgfplotsset{compat=1.12}
\usepackage{mathrsfs}
\usetikzlibrary{arrows}
\usepackage{bm}
\usepackage[alpine]{ifsym}
\usepackage{xpicture}
\usepackage{calculator}
\usepackage{graphicx}
\usepackage{enumerate}
\usepackage{enumitem}
\usepackage{acronym}
\usepackage{xspace}
\usepackage{xfrac}    
\usepackage{faktor} 
\usepackage{caption}
\usepackage{subcaption}

\usepackage[a4paper,top=3.3cm,bottom=3cm,left=3cm,right=3cm,bindingoffset=5mm]{geometry}

\usepackage{float}
\usepackage{color}

\usepackage[colorlinks=true]{hyperref}

\usepackage[normalem]{ulem}

\newtheorem{theorem}{Theorem}[section]
\newtheorem{lemma}[theorem]{Lemma}
\newtheorem{corollary}[theorem]{Corollary}
\newtheorem{proposition}[theorem]{Proposition}

\newtheorem*{A}{Theorem A}
\newtheorem*{B}{Theorem B}
\newtheorem*{TC}{Theorem C}
\newtheorem*{Con}{Conjecture 1}

\theoremstyle{definition}
\newtheorem{example}[theorem]{Example}
\newtheorem{remark}[theorem]{Remark}
\newtheorem{definition}[theorem]{Definition}

\numberwithin{equation}{section}

\newcommand{\Z}{\mathbb{Z}}
\newcommand{\Q}{\mathbb{Q}}
\newcommand{\R}{\mathbb{R}}
\newcommand{\C}{\mathbb{C}}
\newcommand{\PP}{\mathbb{P}}

\newcommand{\pic}{{\rm Pic}}

\def\fd{\mathbb{F}_\delta}

\def\fol{$\mathcal{F}$\xspace}

\def\pf{\mathcal{P}_\mathcal{F}}

\def\kf{K_\mathcal{F}}

\def\conf{\mathcal{C}}
\def\Bf{\mathcal{B}_\mathcal{F}}

\def\cf{\mathcal{C}_\mathcal{F}}

\usepackage{lipsum}

\title[Linear weighted bounded negativity]{Linear weighted bounded negativity}
\author[C. Galindo]{Carlos Galindo}
\address{Universitat Jaume I, Campus de Riu Sec, Departamento de Matem\'aticas \& Institut Universitari de Matem\`atiques i Aplicacions de Castell\'o, 12071
Caste\-ll\'on de la Plana, Spain.}\email{galindo@uji.es}   \email{callejo@uji.es}

\author[F. Monserrat]{Francisco Monserrat}
\address{Universitat Polit\`ecnica de Val\`encia, Institut Universitari de
Matem\`atica Pura i Aplicada \& Departament de Matem\`atica Aplicada,
Camí de Vera s/n, 46022 Val\`encia (Spain).}
\email{framonde@mat.upv.es}

\author[E. P\'erez-Callejo]{Elvira P\'erez-Callejo}

\subjclass[2020]{Primary: 14C20, 14E15}
\keywords{Bounded negativity; Linear bounded negativity; Foliations; Rational surfaces}
\thanks{Partially supported by grants PID2022-138906NB-C22 and PRE2019-089907 funded by MICIU/AEI/10.13039/501100011033 and by ERDF/EU, as well as by grants GACUJIMA/2024/03 and UJI-B2021-02 funded by Universitat Jaume I}
	
\begin{document}

\maketitle

\begin{abstract}
    We propose a linear version of the weighted bounded negativity conjecture. It considers a smooth projective surface $X$ over an algebraically closed field of characteristic zero and predicts the existence of a common lower bound on $C^2/(D\cdot C)$ for all reduced and irreducible curves $C$ and all big and nef divisors such that $D\cdot C>0$, both on $X$.

    We prove that, in the complex case, there exists such a bound for all nef divisors spanning a ray out an open covering of the limit rays of negative curves. In the same vein, we provide explicit bounds when $X$ is a rational surface.

    Our proofs involve the existence of a foliation \fol on $X$ but most of our results are independent of \fol.
\end{abstract}

\section{Introduction}

The \emph{bounded negativity conjecture} (BNC) \cite{Bau2,bauer2013} considers a smooth projective surface $X$ over an algebraically closed field of characteristic zero and predicts the existence of a non-negative integer $B(X)$, which depends only on $X$, satisfying the inequality $C^2\geq-B(X)$ for all reduced and irreducible curves $C$ on $X$. It is a folklore conjecture going back to Enriques and Artin whose study has recently been very active.

Although the trueness of the BNC is unknown in the general case, bounded negativity holds in several specific cases \cite{Harb1,bauer2013,BauPojSch,Ciletal}. The BNC is not true for surfaces $Y$ over algebraically closed fields of positive characteristic \cite[Exercise V.1.10]{Har}, not even if $Y$ is rational as was recently proven in \cite{ChengvanDob}.

The question 6.9 posed by Demailly in \cite{Dem} is one of the reasons explaining the interest in the BNC. The query in our context is if the global Seshadri constant of a surface $X$ as above is positive. A related question (with unknown answer too) asks whether the following infimum
\begin{equation}\label{eqn_intr1}
	\inf\{\epsilon(D,x)\,|\,D\in\pic(X)\text{ is ample}\}
\end{equation}
is positive for any fixed point $x\in X$, where $\epsilon(D,x)$ stands for the Seshadri constant of $D$ at $x$.

In view of the difficulty of the BNC, the literature proposes some variants. A \emph{weak version of the BNC} states that $C^2\geq -B(X,g)$, where $B(X,g)$ is a non-negative integer depending on $X$ and a positive integer $g$, and the inequality holds for every reduced curve $C$ on $X$ such that the geometric genus of its irreducible components is bounded by $g$.

This weak version was conjectured in \cite{Bau2,bauer2013} and proved in \cite{Hao}. Very recently further developments around the weak BNC can be found in \cite{CilFont} and \cite{MisRay}.

The so-called \emph{weighted BNC} is getting a lot of interest lately. It was stated in \cite[Conjecture 3.7.1]{Bau2} and predicts the existence of a non-negative integer $B_W(X)$, which depends only on $X$, such that {$C^2\geq -(D\cdot C)^2\cdot B_W(X)$} for each reduced and irreducible curve {$C$} and each big and nef divisor $D$ on $X$ such that $D\cdot C >0$. It is worthwhile mentioning that, if this conjecture is true, the infimum defined in (\ref{eqn_intr1}) is always positive \cite[Proposition 3.7.2]{Bau2}. Some progress towards the weighted BNC can be consulted in \cite{Laface,GalMonMorPerC2022,GalMonMor3,CilFont,MisRay,Wang}.

The weighted BNC prognosticates the existence of a bound on the self-inter\-sec\-tions $C^2$ (where $C$ is as before) quadratically running over $D\cdot C$, where $D$ is any nef and big divisor on $X$ such that $D\cdot C>0$.
However, references as \cite{Laface,CilFont,MisRay} seem to indicate the possibility of having a linear bound. Thus we propose the following \emph{linear weighted bounded negativity conjecture} (LWBNC).

\begin{Con}\label{cnj1}
	Let $X$ be a smooth projective surface over an algebraically closed field of characteristic zero. Then, there exists a non-negative integer $B_W(X)$, depending only on $X$, such that \[
	\frac{C^2}{D\cdot C}\geq -B_W(X)\]
	for all reduced and irreducible curves $C$ on $X$ and all big and nef divisors $D$ on $X$ such that $D\cdot C>0$.
\end{Con}

We devote this paper to provide some results related with Conjecture \ref{cnj1} in the complex case. For simplicity, from now on, the word \emph{surface} means smooth complex projective surface. We divide this article in two parts.

The first one corresponds to Section \ref{sec_lwbn}, and studies linear weighted bounded negativity for any surface $X$. Our main result is to prove the existence of a common bound for the quotients $C^2/(D\cdot C)$ for any nef divisor $D$ spanning a ray out of a fixed open covering of the set of limit rays of negative curves, and for any negative curve $C$ such that $D\cdot C>0$.

The second part, developed in Section \ref{sec_explbound}, treats linear weighted bounded negativity for (complex) rational surfaces $S$. Here, we provide explicit bounds which depend on the configuration of infinitely near points and the relatively minimal surface (the projective plane or a Hirzebruch surface) considered to obtain $S$.

Before giving more details about our main findings, we must indicate that we consider foliations on surfaces as a new tool in this context. We will assume (or prove that there exists) a (suitable) foliation on $X$. Although this foliation is crucial in our development, many of our bounds do not depend on it.

In the last decades, new ideas and concepts have been introduced in Algebraic Geometry, many of them for studying positivity. They have led to great advances in the minimal model program \cite{BirCasHacMcKer} and subsequently have been applied to the field of foliations \cite{Bru,CHLX,CasSpic}. We are only interested in foliations on surfaces $X$ and in Section \ref{sec_lwbn}, it will be crucial the existence of a foliation on $X$ without algebraic invariant curves. On the contrary, algebraically integrable foliations on the complex projective plane $\PP^2$ or a complex Hirzebruch surface $\fd$ are instrumental in Section \ref{sec_explbound}. To decide about algebraic integrability in this context is a classical open problem which comes back to the XIX century. Darboux, Poincaré, Painlevé and Autonne studied this question with the aim to solve planar polynomial differential systems.

To decide whether a foliation admits a rational first integral is possible in certain situations \cite{pr-si,d-l-a,GalMon2006,fergia,GalMon2014,bost} and algebraically integrable foliations are present in very different contexts \cite{hew,lli2,lli,CHLX}.

In the setting of foliations, Theorem \ref{thm_attch_fol} is the main result in the paper. It fixes any configuration $\conf$ of infinitely near points over $\PP^2$ or $\fd$ and shows the existence of an algebraically integrable foliation \fol such that all the points in $\conf$ are infinitely near dicritical singularities of $\mathcal{F}$. It also gives a bound for the degree or bidegree of both, \fol and its rational first integral (which is that of a reduced and irreducible invariant generic curve). In the case of the projective plane, the degree of this first integral is, in fact, computed.

Results in Section \ref{sec_lwbn} of the paper are stated for any surface $X$. Consider a nef divisor $D$ on $X$ and define \[
\nu_D(X):=\inf\left\{\frac{C^2}{D\cdot C}\;|\; C\text{ is a negative curve on }X\text{ such that }D\cdot C>0\right\}.\]
A first interesting result in this section is the following.

\begin{A}[Theorem \ref{main1} in the paper]
	Let $X$ be a (smooth complex projective) surface and $D$ a nef and big divisor on $X$. Then, $\nu_D(X)$ is a real number, that is, it cannot be $-\infty$.
\end{A}

For the proof of Theorem A, we polarize $X$ with an ample divisor $H$ which is a canonical divisor of a foliation on $X$ without invariant curves.

Then, considering the real vector space $\rm{NS}_{\R}(X)=\rm{NS}(X)\otimes\R$ and standing $[B]$ for the class of a divisor $B$ on ${\rm NS}_{\mathbb{R}}(X)$, we define \[
\mathcal{A}(X):=\{\text{rays }\mathbb{R}_+x\,|\, x\in {\rm NS}_{\mathbb{R}}(X)\text{ and }[H]\cdot x>0\},\]
\[
\mathcal{R}(X):=\{\text{rays }\mathbb{R}_+[C]\subseteq \mathcal{A}(X)\,|\, C \text{ is a negative curve on }X\},\]
and $\mathcal{R}^{ac}(X)$, which is the set of accumulation points of $\mathcal{R}(X)$ in $\mathcal{A}(X)$. With these ingredients the \emph{main result} in the paper approaching the LWBNC is the next one.

\begin{B}[Theorem \ref{thm_cota_paco} in the paper]
	Let $X$ be a (smooth complex projective) surface. Fix an open covering $\{U_\alpha\}_{\alpha\in \Lambda}$ of the above introduced set $\mathcal{R}^{ac}(X)$. Then, there exists a positive real number $\eta$ such that
	$$\frac{C^2}{D\cdot C}\geq -\eta$$
	for any nef divisor $D$ on $X$ such that the ray $\mathbb{R}_+[D]$ does not belong to $\bigcup_{\alpha\in \Lambda} U_\alpha$ and for any negative curve $C$ on $X$ such that $D\cdot C>0$.
\end{B}

The above bound $-\eta$ is not explicit. However, in Section \ref{sec_explbound} we are able to get explicit bounds when considering rational surfaces. Denote by $S_0$ either the complex projective plane $\PP^2$ or a complex Hirzebruch surface $\fd$, $\delta\geq 0$. Any rational surface $S$ can be obtained by blowing-up $S_0$ at a configuration of infinitely near points $\conf$  over $S_0$. The triple $(S,S_0,\conf)$ is named an $S_0$-tuple. Denote by $L$ a line on $\PP^2$ and by $F$ and $M$ a fiber and a section of $\fd$ as we will describe in Subsection \ref{subsec_rat_surf}. We also set $E_q$ the strict transform on $S$ of the exceptional divisor obtained by blowing-up $q\in C$.

Subsection \ref{sec_41} assumes the existence and knowledge of a foliation \fol on $S_0$ such that every point in $\conf$ is an infinitely near singularity of \fol. Then, an explicit bound on $\nu_D(S)$, depending on \fol, is given for a large set $\Delta$ of nef divisors on $S$, see Corollary \ref{a}. Given a divisor $G$ on $S$ and a positive real number $\epsilon$, we define
\begin{multline*}
	\Delta(X;G,\epsilon):=\{\text{nef }\R\text{-divisors $D$ on }S\;|\;(D-\epsilon G)\cdot C\geq 0\text{ for all}\\
	\text{ reduced and irreducible curve }C\text{ on }S\text{ such that }D\cdot C>0\},
\end{multline*}
and the set $\Delta$ in Corollary \ref{a} is as before with $G$ equal to $L^*$ (respectively, $F^*+M^*$) whenever $S_0=\PP^2$ (respectively, $\fd$).

In Subsection \ref{sec_42} we state and prove our main result on rational surfaces. It is a consequence of the forthcoming Theorem \ref{teo2}, where we auxiliary use the before mentioned Theorem \ref{thm_attch_fol}. Next, we state this result.

\begin{TC}[Corollary \ref{cor_expl1} in the paper]
	Let $(S,S_0,\conf)$ be an $S_0$-tuple and let $\epsilon$ be a positive real number. Set $n=\# \conf$ and let $d$ be the positive integer defined in (\ref{positiveinteger}); it can be easily obtained from the proximity graph of the configuration $\conf$. Then setting $\gamma:=\max\{-E_q^2\,|\,q\in\conf\}$, which can also be obtained from the proximity graph, it holds that:
	\begin{enumerate}
		\item If $S_0=\PP^2$, then \[
		\nu_D(S)\geq\min \left\{\frac{1}{\epsilon}(3-2d),\frac{1}{\epsilon}d(1-n),-\gamma\right\}\]
		for all divisor $D\in\Delta(S;L^*,\epsilon)$.
		\item If $S_0=\fd$, $\delta\geq 0$, then \[
		\nu_D(S)\geq \min \left\{\frac{1}{\epsilon}(2-2d-\delta),\frac{1}{\epsilon}(-\delta-2)dn,-\gamma\right\}\]
		for all divisor $D\in\Delta(S,F^*+M^*,\epsilon)$.
	\end{enumerate}
\end{TC}

We conclude the introduction with an outline of the organization of the paper. Section \ref{secPrelim} contains some preliminaries that will be necessary for our development; here, very briefly we treat configurations of infinitely near points, Hirzebruch surfaces and foliations on a surface. Section \ref{sec_lwbn} contains our results on linear weighted bounded negativity for any surface, while the case of rational surfaces is studied in Section \ref{sec_explbound}.

\section{Preliminaries} \label{secPrelim}

We devote this section to recall some concepts and results we will need to develop this paper. Remind that, for us, surface means smooth complex projective surface. We are going to study linear weighted bounded negativity on any surface on the one hand and, on the other hand, whenever the involved surface $S$ is rational, we will give specific bounds on $\nu_D(S)$ for a large range of nef divisors $D$ on $S$.

In our development, foliations on surfaces and the reduction of their singularities are an important tool. The reduction is obtained through successive blowups of surfaces containing the foliation and its transforms. The blowup centers constitute a configuration of infinitely near points (see \cite{Seiden} and \cite[Chapter 1]{Bru}). Our first subsection recalls some definitions and facts on configurations of infinitely near points.

\subsection{Configurations of infinitely near points}\label{subsec_conf}

Let $p$ be a point on a surface $X$. Denote by $\pi_p:{\rm Bl}_p(X)\longrightarrow X$ the blowup of $X$ at $p$. The points lying in the exceptional divisor $E_{p}=\pi_p^{-1}(p)$ are referred to as the points in the \emph{first infinitesimal neighborhood} of $p$. We use induction to define the $i$\emph{th infinitesimal neighbourhood of} $p$, for any $i\geq 1$. A point $q$ is said to be \emph{infinitely near} $p$ (denoted $q\geq p$) if either $q=p$ or $q$ belongs to some $k$th infinitesimal neighborhood of $p$. When $p<q$ we say that $p$ precedes $q$. Lastly, a point $q$ is \emph{infinitely near} $X$ whenever it is infinitely near some point on $X$. Furthermore, given two points $p$ and $q$ both infinitely near $X$, we say that $q$ \emph{is proximate to} $p$ (denoted $q\rightarrow p$) if $q$ lies on $E_p$ or any of its strict transforms. If $q$ is proximate to $p$ but does not belong to $E_p$, then $q$ is called \emph{satellite}; otherwise, it is \emph{free}.

 Let us consider a finite sequence of blowups
 \begin{equation}\label{eqn_def_blw}
 	\pi:X_{n}\overset{\pi_n}{\longrightarrow}X_{n-1}\overset{\pi_{n-1}}{\longrightarrow}\cdots\overset{\pi_2}{\longrightarrow}X_1\overset{\pi_1}{\longrightarrow}X_0=X,
 \end{equation}
where $\pi_{i}$ is the blowup centered at a closed point $p_i\in X_{i-1},\,1\leq i \leq n$. The set of centers $\conf=\{p_i\}_{i=1}^n$ is called a (non-empty) \emph{configuration (or cluster)} (of infinitely near points) over $X$ \cite{Camp1996,Cas} and the obtained surface $X_n$ the \emph{sky} of $\conf$. We denote by $E_{p_i}$ (or just $E_i$) the exceptional divisor created after blowing-up at the point $p_i$. For simplicity, we also stand $E_{p_i}$ for the strict transform of $E_{p_i}$ on $X_n$. Throughout this paper, whenever we consider a configuration $\conf=\{p_i\}_{i=1}^n$, we will assume that each subindex $i$  corresponds to the order of appearance of $p_i$ in the sequence (\ref{eqn_def_blw}).

The \emph{level} of $q\in\conf$, $l(q)$, is the minimum number of blowups one needs to obtain the surface containing $q$. We say that $q\in\conf$ is an \emph{origin} (respectively, an \emph{end}) of $\conf$ if $q\in X$ (respectively, if $\conf$ contains no proximate to $q$ point). Notice that the origins are the points of level $0$. The set of origins (respectively, ends) of $\conf$ is denoted by $O_{\mathcal{C}}$ (respectively, $\mathcal{E}_{\conf}$). Let us define the sets $(\conf)_q:=\{q'\in\conf|q\leq q'\}$ and $(\conf)^q:=\{q'\in\conf|q'\leq q\}$ and notice that \[
\conf=\bigcup_{q\in O_{\conf}}(\conf)_q=\bigcup_{q\in \mathcal{E}_{\conf}} (\conf)^q.\]

 The points in a configuration $\conf$ are partially ordered by the ordering $\leq $. It is often useful to consider the so-called \emph{proximity graph} of $\conf$. Its vertices correspond to the points in $\conf$ and an edge joints two vertices $q$ and $q'$ if one is proximate to the other. For a better readability, we omit those edges that can be deduced from others. When representing the graph, we arrange the vertices in ascending order according to their levels, see the forthcoming Figure \ref{fig_ex}. Note that proximity and dual graphs of a configuration $\conf$ are equivalent combinatorial objects.

 One can use a matrix to encode the proximity relation. It is named the \emph{proximity matrix} of $\conf$, $P_{\conf}=(p_{ij})$, $1\leq i,j\leq n$, and it is defined by
 \begin{equation}\label{eqn_prox_mtx}
 p_{ij}=\left\lbrace\begin{array}{ll}
 1 &\text{if }i=j,\\
 -1&\text{if }p_i\rightarrow p_j,\\
 0&\text{otherwise.}
 \end{array}\right.
\end{equation}

 As a consequence of \cite[Theorem V.10]{Beau}, any rational surface can be seen as the sky of a configuration over either the projective plane $\PP^2$ or a Hirzebruch surface $\fd$, for $\delta\neq 1$. Our next subsection briefly recalls well-known facts on Hirzebruch surfaces we will use later.

\subsection{Hirzebruch surfaces}\label{subsec_rat_surf}


Denote by $\PP^1$ the complex projective line. For any non-negative integer $\delta$, the $\delta$th Hirzebruch surface is the projective space $\fd:=\PP(\mathcal{O}_{\PP^1}\oplus\mathcal{O}_{\PP^1}(\delta))$. It is a geometrically ruled surface over $\PP^1$, whose corresponding surjective morphism we denote by $\sigma:\fd\longrightarrow\PP^1$. It can be regarded as the quotient $(\C^2\setminus\{(0,0)\})\times(\C^2\setminus\{(0,0)\})$ by an action of the algebraic torus $\C^*\times\C^*$. This action is defined as follows: For each $(\lambda,\mu)\in\C^*\times\C^*$, $(\lambda,\mu):(X_0,X_1;Y_0,Y_1)\rightarrow(\lambda X_0,\lambda X_1,\mu Y_0,\lambda^{-\delta}\mu Y_1)$. Its homogeneous coordinate ring is $\C[X_0,X_1,Y_0,Y_1]$ where the variables $X_0,\,X_1,\,Y_0$ and $Y_1$ are graded on $\Z\times\Z_{\geq 0}$ with values $(1,0),\,(1,0),\,(0,1)$ and $(-\delta,1)$, respectively. We say that a polynomial in $\C[X_0,X_1,Y_0,Y_1]$ is \emph{bihomogeneous of bidegree} $(d_1,d_2)\in\Z^2$ if it belongs to the homogeneous piece corresponding to $(d_1,d_2)$, i.e., it is a linear combination of monomials $X_0^{a_0}X_1^{a_1}Y_0^{b_0}Y_1^{b_1}$ with $a_0+a_1-\delta b_1=d_1$ and $b_0+b_1=d_2$.

Let $F$ denote a fiber of $\sigma$ and $M$ (respectively, $M_0$) the section of $\sigma$ given by the equation $Y_0=0$ (respectively, $Y_1=0$). Notice that $M^2=\delta$ and $M_0^2=-\delta$. $M_0$ is referred to as the \emph{special section} of $\fd$. Except for the case when $\delta=0$, where $M_0$ and $M$ are linearly equivalent, $M_0$ is the only reduced and irreducible curve on $\fd$ with negative self-intersection. The group $\pic(\fd)$ is isomorphic to $\Z\oplus\Z$ and it is generated by the classes $[F]$ and $[M]$ of $F$ and $M$. $\fd$ can be covered by four isomorphic to $\C^2$ affine open sets, defined by \[
U_{ij}=\{(X_0,X_1,Y_0,Y_1)\in\fd|X_i\neq 0 \text{ and }Y_j\neq 0\},\]
for $i,j\in\{1,0\}$ (see, for instance, \cite{GalMonOliv} for more details). This behaviour is close to that of the complex projective plane $\PP^2$, which can be covered by three affine open sets, defined by $U_A=\{(X:Y:Z)\in\PP^2|A\neq 0\}$, for $A\in \{X,Y,Z\}$, $X, Y$ and $Z$ being projective coordinates. Each one of these affine open sets can be identified with $\C^2$.

Throughout the paper, we denote by \(\deg (C)\) (respectively, \((\deg_1 (C), \deg_2 (C))\)) the degree (respectively, bidegree) of a curve $C$ in \(\mathbb{P}^2\) (respectively, a Hirzebruch surface). We use the same notation  for the degree (respectively, bidegree) of homogeneous (respectively, bihomogeneous) polynomials in $\C[X,Y,Z]$ (respectively, $\C[X_0,X_1,Y_0,Y_1]$.


Foliations on surfaces are an important tool in this paper. The next subsection introduces this concept and makes a short review of some of their properties.

\subsection{Foliations}\label{subsec_foliation}

\subsubsection{Definition and invariant curves} Let $X$ be a surface. A \emph{singular holomorphic foliation with isolated singularities}\label{subsubinv} (\emph{foliation} in the sequel) \fol on $X$ can be defined by a family of pairs $\{(U_i,v_i)\}_{i\in I}$, where $\{U_i\}_{i\in I}$ is an open covering of $X$ and $v_i$ is a holomorphic vector field on $U_i$ with isolated zeroes such that, for any $i,j\in I$, $v_i=g_{ij}v_j$ on $U_i\cap U_j$ for some element $g_{ij}\in\mathcal{O}_X(U_i\cap U_j)^*$, $\mathcal{O}_X$ being the sheaf of holomorphic functions on $X$. We identify the foliations given by $\{(U_i,v_i)\}_{i\in I}$ and $\{(U'_i,v'_i)\}_{i\in I}$ whenever $v_i$ and $v'_i$ coincide on $U_i\cap U'_i$ up to multiplication by a nowhere vanishing holomorphic function.

The functions $g_{ij}$ define a multiplicative cocycle and, hence, a cohomology class in $H^1(X,\mathcal{O}^*_X)$ (that is, a line bundle on $X$). The sheaf associated to this line bundle is denoted by $\mathcal{K}_{\mathcal{F}}$ and called \emph{canonical sheaf} of \fol; moreover a divisor $K_{\mathcal{F}}$ on $X$ such that $\mathcal{O}_X(K_{\mathcal{F}})=\mathcal{K}_{\mathcal{F}}$ is a \emph{canonical divisor} of $\mathcal{F}$. Equivalently, a foliation is a global section of $\Theta_X\otimes \mathcal{K}_{\mathcal{F}}$ (modulo multiplication by an element of ${\mathcal O}_X^*(X)$), where $\Theta_X$ the tangent sheaf of $X$.

A foliation on $X$ can also be defined by using $1$-forms; see \cite{Bru}, for instance.

For any point $p\in X$, let $v_p$ be a local holomorphic vector field defining a foliation \fol around $p$. The \emph{algebraic multiplicity} of \fol at  $p$, denoted by $\nu_p(\mathcal{F})$, is the order of $v_p$ at $p$, i.e., the integer $m$ such that $v_p\in\mathfrak{m}_p^m\Theta_{X,p}$ and $v_p\notin\mathfrak{m}_p^{m+1}\Theta_{X,p}$, $\mathfrak{m}_p$ being the maximal ideal of the local ring $\mathcal{O}_{X,p}$. The \emph{singularities} of \fol are the points $p\in X$ where $\nu_p(\mathcal{F})>0$.

A reduced  curve $G$ on $X$ is \fol-\emph{invariant} if, for all $q\in G$, $v_q({f_q})\in \langle f_q\rangle$, where $f_q\in {\mathcal O}_{X,q}$ defines a local equation of $G$ at $q$.

Let $C$ be a reduced curve on $X$ whose irreducible components are not \fol-invariant. Following \cite[Chapter 2, Section 2]{Bru}, we define the \emph{tangency order of \fol to $C$ at a} point $p\in C$ as \[
{\rm tang}(\mathcal{F},C,p):=\dim_{\C}\mathcal{O}_{X,p}/\langle f_p, v_p(f_p)\rangle,\]
$f_p$ being a local equation of $C$ at $p$. Notice that $\mathcal{O}_{X,p}/\langle f_p, v_p(f_p)\rangle$ is a finite-dimensional vector space over $\C$ since $v_p(f_p)\notin \langle f_p\rangle$. Moreover, if $\mathcal{F}$ is transverse to $C$ at $p$ (it means that every local invariant curve of the foliation and $C$ meet transversely), ${\rm tang}(\mathcal{F},C,p)=0$. As the irreducible components of $C$ are not \fol-invariant, there are finitely many points where $\mathcal{F}$ is not transverse to $C$. Hence we define \[
{\rm tang}(\mathcal{F},C):=\sum_{p\in C}{\rm tang}(\mathcal{F},C,p).\]

The following lemma (which is a key fact in the proofs of our results) follows from the inequality $\text{tang}(\mathcal{F},C)\geq 0$ and \cite[Proposition 2.2]{Bru}, which states that \[
C^2=-K_{\mathcal{F}}\cdot C+\text{tang}(\mathcal{F},C),\]
where $K_{\mathcal{F}}$ is a canonical divisor of \fol.

\begin{lemma}\label{lemma_bncota}
Let $\mathcal{F}$ be a foliation defined on a  surface $X$. If $C$ is a reduced curve on $X$ whose irreducible components are not $\mathcal{F}$-invariant, then
\begin{equation}\label{eqn_neqac}
C^2\geq -K_{\mathcal{F}}\cdot C.
\end{equation}
\end{lemma}

In the last part of this subsection, we remind some facts concerning foliations on $\PP^2$ or $\fd$. We express these foliations in terms of homogeneous 1-forms as in \cite{Campiol,GalMonOliv}.

As before said, a foliation \fol on $\PP^2$ is given by a global section of $\Theta_{\PP^2}\otimes\mathcal{O}_{\PP^2}(r-1)$, for some $r\in\Z_{\geq 0}$, which is called the \emph{degree} of \fol; $\kf=(r-1)L$ is a canonical divisor of \fol, $L$ being the divisor of a line. Moreover, \fol is uniquely determined by a homogeneous $1$-form $\Omega:=AdX+BdY+CdZ$, where $A,B,C\in\C[X,Y,Z]$ are homogeneous polynomials of degree $r$ without common factors and such that $AX+BY+CZ=0$ (see \cite{Campiol}).

Analogously, a foliation \fol on $\fd$, $\delta\in\Z_{\geq 0}$, is given by a global section of $\Theta_{\fd}\otimes\mathcal{O}_{\fd}(r_1,r_2)$, $(r_1,r_2)\in\Z^2$. This last pair is the \emph{bidegree} of \fol. Therefore a canonical divisor of \fol is $\kf=r_1F+r_2M$, where $F$ and $M$ are the divisors defined in Subsection \ref{subsec_rat_surf}. As can be seen in \cite{GalMonOliv}, the foliation \fol on $\fd$ is uniquely determined by a bihomogeneous $1$-form $\Omega:=A_0dX_0+A_1dX_1+B_0dY_0+B_1dY_1$, where $A_0,A_1,B_0,B_1\in\C[X_0,X_1,Y_0,Y_1]$ are bihomogeneous polynomials of bidegrees \[
(r_1-\delta+1,r_2+2),\,(r_1-\delta+1,r_2+2),\,(r_1-\delta+2,r_2+1)\text{ and }(r_1+2,r_2+1),\]
respectively, without common factors, and such that $A_0X_0+A_1X_1-\delta B_1Y_1=0$ and $B_0Y_0+B_1Y_1=0$.

In this paper we use the symbol $S_0$ to represent either the complex projective plane $\PP^2$ or a complex Hirzebruch surface $\fd$. Let $\Omega$ be the $1$-form defining a foliation \fol on $S_0$. We say that \fol \emph{is algebraically integrable} if there exists a rational function $R$ of $S_0$, called a \emph{rational first integral (of \fol)}, such that $\Omega\wedge dR=0$.

If \fol is an algebraically integrable foliation on $S_0=\PP^2$ (respectively, $S_0=\fd$) with rational first integral $R=\frac{F}{G}$, then there exists a unique irreducible pencil of curves on $S_0$, $\pf=\langle F,G\rangle$, such that all the reduced and irreducible \fol-invariant curves are exactly the reduced and irreducible components of the curves in $\pf$. Let $BP(\mathcal{P}_\mathcal{F})$ represent the {\it configuration (or cluster) of base points} of $\pf$ (see \cite[Section 7.2]{Cas} or the proof of \cite[Theorem II.7]{Beau}).

\subsubsection{Reduction of singularities}

Let \fol be a foliation on a surface $X$. Let $\{x,y\}$ be local coordinates around a point $p\in X$ and assume that \fol at $p$ is given by a local differential $1$-form $\omega=a(x,y)dx+b(x,y)dy=\sum_{i\geq 0}\omega_i$, where $\omega_i=a_i(x,y)dx+b_i(x,y)dy$ denotes the homogeneous component of degree $i$ of $\omega$. Notice that $\omega_i=0$ for $i< \nu_p(\mathcal{F})$ and $p$ is a singularity of \fol if and only if $\omega_0=0$. We say that $p$ is a \emph{simple singularity} whenever $\nu_p(\mathcal{F})=1$ and the matrix \[
\left(
\begin{array}{cc}
	\frac{\partial b_1}{\partial x} & \frac{\partial b_1}{\partial y} \\
	- \frac{\partial a_1}{\partial x} & - \frac{\partial a_1}{\partial y}\\
\end{array}
\right)
\]
has two eigenvalues $\lambda_1, \lambda_2$ such that either $\lambda_1 \lambda_2\neq 0$ and $\lambda_1/\lambda_2\notin\Q_{>0}$, or $\lambda_1 \lambda_2 = 0$ and $\lambda_1^2 + \lambda_2^2\neq 0$. Nonsimple singularities are named to be \emph{ordinary}. A singularity $p$ of \fol is called \emph{dicritical} if there are infinitely many separatrices (holomorphic irreducible and \fol-invariant curves defined on a neighbourhood of $p$) passing through $p$.

In \cite{Seiden}, Seidenberg proved that the ordinary singularities can be removed by blowing-up:

\begin{theorem}\label{seiden}
	Let \fol be a foliation on a surface $X$. Then, there exists a (minimal) sequence of blowups, \[
	\pi:=\tilde{X}=X_{n}\overset{\pi_n}{\longrightarrow}X_{n-1}\overset{\pi_{n-1}}{\longrightarrow}\cdots\overset{\pi_2}{\longrightarrow}X_1\overset{\pi_1}{\longrightarrow}X_0=X,\]
		such that the strict transform of \fol on $\tilde{X}$ has no ordinary singularity.
\end{theorem}

Let \fol and $\pi$ be as in Theorem \ref{seiden}. The set $\conf_\mathcal{F}$ of blowup centers of $\pi$ constitutes the \emph{singular configuration} of \fol. Elements in $\conf_{\mathcal{F}}$ are the so-called \emph{infinitely near singularities} of \fol. Moreover, we say that $p\in\conf_\mathcal{F}$ is an \emph{infinitely near dicritical singularity} if it is a dicritical singularity of the strict transform of the foliation on the surface containing $p$. We denote by $\Bf$ the {\it set of infinitely near dicritical singularities of} \fol.

 The next result is deduced from \cite{TesJulio}. It is proved in \cite[Theorem 2]{fergal1} for foliations on the projective plane but it can be easily adapted to foliations on a Hirzebruch surface $\fd$ since its proof is local. Recall that we use $S_0$ to represent either $\PP^2$ or $\fd$, $\delta\geq 0$.

\begin{proposition}\label{prop_bfbp}
 	Let \fol be a foliation on $S_0$ having a rational first integral $R=\frac{F}{G}$ and associated pencil $\pf=\langle F,G\rangle$. Then \[
 	\Bf=BP(\mathcal{P}_{\mathcal{F}}).\]
 \end{proposition}

Our next section is devoted to approach the LWBNC.

\section{Linear weighted bounded negativity. General setting}\label{sec_lwbn}

Let $X$ be a (smooth complex projective) surface. Let us denote by ${\rm NS}(X)$ the Néron-Severi group of $X$ \cite[Definition 1.1.15]{Laz1} and consider the real vector space ${\rm NS}_{\mathbb{R}}(X):={\rm NS}(X)\otimes \mathbb{R}$. Its dimension is $\rho(X)$, the Picard number of $X$. As usual, the symbol $\cdot$ stands for the intersection product. For any $\R$-divisor $D$ on $X$, $[D]$ denotes the image of $D$ in ${\rm NS}_{\mathbb{R}}(X)$. Set ${\rm Eff}(X)$ (respectively, ${\rm Nef}(X)$) the \emph{effective cone} (respectively, \emph{nef cone}) of $X$, that is, the convex cone of ${\rm NS}_{\mathbb{R}}(X)$ spanned by the images of the effective (respectively, nef) classes in ${\rm NS}(X)$.

In this paper, a \emph{negative curve} on $X$ is a reduced and irreducible curve $C$ on $X$ such that $C^2<0$.

\medskip

For any nef divisor $D$ on $X$, we define
\begin{align}\label{eqn_nuD}
\nu_{D}(X):=&\inf\left\lbrace\frac{C^2}{D\cdot C}\;|\;C \text{ is a negative curve on }X\text{ such that }D\cdot C>0\right\rbrace.
\end{align}

Notice that this value belongs to  $\{-\infty\}\cup \mathbb{R}$. Later, in Theorem \ref{main1}, we will show that, when $D$ is nef and big, $\nu_D(X)$ is in fact a real number.\medskip

For any divisor $G$ on $X$ and for any positive real number $\epsilon$, we consider the following set of \emph{nef $\R$-divisors} on $X$:
\begin{multline}\label{eqn_Delta}
\Delta(X;G,\epsilon):=\{\text{nef }\R\text{-divisors $D$ on }X\;|\;(D-\epsilon G)\cdot C\geq 0\text{ for all}\\
	\text{ reduced and irreducible curve }C\text{ on }S\text{ such that }D\cdot C>0\}.
\end{multline}
In particular, a nef $\mathbb{R}$-divisor $D$ belongs to $\Delta(X;G,\epsilon)$ if the $\mathbb{R}$-divisor $D-\epsilon G$ is nef.\medskip

Our first result in this section proves that, if $X$ is equipped with a foliation $\mathcal{G}$  having finitely many negative $\mathcal{G}$-invariant curves, then  the value $\nu_D(X)$ is finite. Moreover, $\mathcal{G}$ allows us to bound $\nu_D(X)$ whenever $D\in\Delta(X;K_{\mathcal{G}},\epsilon)$ for some $\epsilon>0$, $K_{\mathcal{G}}$ being a canonical divisor of $\mathcal{G}$. In fact, our bound can be explicitly computed from the knowledge of the set of negative $\mathcal{G}$-invariant curves on $X$.\medskip

Let $\mathcal{G}$ be a foliation on $X$. We denote by $A(\mathcal{G})$ the set of negative $\mathcal{G}$-invariant curves on $X$.
Whenever $A(\mathcal{G})$ is finite, let us define
\[
\alpha(\mathcal{G}):= \left\{\begin{array}{ll}
		\max\{-W^2\mid W\in A(\mathcal{G})\} &\text{if }A(\mathcal{G})\neq \emptyset,\\
		0&\text{otherwise}.
	\end{array}\right.
	\]	
Also in this case, for any nef divisor $D$ on $X$, set $$\alpha_D(\mathcal{G}):=\max \{-W^2/(D\cdot W)\mid W\in A(\mathcal{G}) \mbox { and } D\cdot W>0\}$$ if $A(\mathcal{G})$ is not empty, and $\alpha_D(\mathcal{G}):=0$ otherwise.

Let us state and prove the before mentioned first result for which we keep the above notation.

\begin{theorem}\label{foliated1}
Let $\mathcal{G}$ be a foliation on $X$ such that $A(\mathcal{G})$ is finite and let $\epsilon$ be a positive real number. If $D$ is a nef divisor on $X$ such that $D\in \Delta(X,K_{\mathcal{G}},\epsilon)$, then
\begin{equation}\label{eqa}
\frac{C^2}{D\cdot C}\geq -\frac{1}{\epsilon}
\end{equation}
 for any negative curve $C$ on $X$ such that $C\not\in A(\mathcal{G})$ and $D\cdot C>0$. As a consequence,
$$\nu_D(X)\geq \min\{-1/\epsilon,-\alpha_D(\mathcal{G})\}\geq  \min\{-1/\epsilon,-\alpha(\mathcal{G})\}.$$

\end{theorem}

\begin{proof}

Let $C$ be a negative curve on $X$ such that $C\not\in A(\mathcal{G})$ and $D\cdot C>0$. Then
$$\frac{C^2}{D\cdot C}\geq -\frac{K_{\mathcal{G}}\cdot C}{D\cdot C}\geq -\frac{1}{\epsilon},$$
where the first inequality is a consequence of Inequality (\ref{eqn_neqac}) and the second one follows from the fact that $D\in \Delta(X,K_{\mathcal{G}},\epsilon)$.\medskip

The last assertion in the statement is straightforward from the first one.

\end{proof}

Fix an ample line bundle $\mathcal L$ on $X$ and consider a positive integer $r$ and a global section $s$ of the twisted tangent sheaf $\Theta_X\otimes \mathcal{L}^r$ such that $s$ defines a foliation $\mathcal{F}$ on $X$ without algebraic invariant curves. Notice that $r$ and $s$ exist by \cite{Coutinho}. Let $H$ be a divisor on $X$ such that ${\mathcal L}^r={\mathcal O}_X(H)$. Notice that $H$ is a canonical divisor of $\mathcal{F}$.\medskip

 \emph{From now on, and until the end of this section, we will consider the polarized surface $(X,H)$}.\medskip

Since the foliation $\mathcal{F}$ has no algebraic invariant curve, the above defined set $A(\mathcal{F})$ is empty. Therefore, Theorem \ref{foliated1} allows us to deduce the following result.

\begin{corollary}
Let $X$ and $H$ be as above. Consider a positive real number $\epsilon$ and let $D$ be a nef divisor on $X$ such that $D\in \Delta(X;H,\epsilon)$. Then,
$$\nu_D(X)\geq -\frac{1}{\epsilon}.$$
\end{corollary}

The following result shows that \emph{the value $\nu_D(X)$ is finite for any nef and big divisor $D$ on a surface $X$}.

\begin{theorem}\label{main1}
Let $D$ be a nef and big divisor on $X$. Then, there exists a positive real number $B(D)$ (depending on $D$) such that
$$\frac{C^2}{D\cdot C}\geq -B(D)$$
for any negative curve $C$ on $X$ such that $D\cdot C>0$. That is, $\nu_D(X)\in \mathbb{R}$ for all nef and big divisor $D$ on $X$.
\end{theorem}

\begin{proof}
Let $C$ be a negative curve on $X$ such that $D\cdot C>0$.
Since $D$ is big, by \cite[Corollary 2.2.7]{Laz1}, there exist a positive integer $m$ and an effective divisor $N$ on $X$ such that $[mD]=[H]+[N]$, $H$ being the above mentioned ample divisor polarizing $X$. Then $(m+1)[D]=([D]+[H])+[N]$ and, therefore,
$$[D]=[A]+\frac{1}{m+1}[N],$$
where $A$ is the $\mathbb{Q}$-divisor $\frac{1}{m+1}(D+H)$. Notice that $A\in \Delta(X;H,\frac{1}{m+1})$ because $D$ is nef (see the definition given in (\ref{eqn_Delta})) and, therefore,
$$\frac{A\cdot C}{H\cdot C}\geq \frac{1}{m+1}.$$
If $C$ is not a component of $N$, it holds that
$$\frac{D\cdot C}{H\cdot C} =\frac{A\cdot C}{H\cdot C}+\frac{1}{m+1}\; \frac{N\cdot C}{H\cdot C}\geq \frac{1}{m+1}$$
and, by Lemma \ref{lemma_bncota},
$$\frac{C^2}{D\cdot C}\geq -\frac{H\cdot C}{D\cdot C}\geq -(m+1).$$
Since the set of reduced and irreducible components of $N$ is finite, this concludes the proof.
\end{proof}

The next two lemmas are focused to prove the forthcoming Theorem \ref{thm_cota_paco}, which is one of the main results in this paper.

The surface $X$ and the divisor $H$ keep to be as above. For each $x\in {\rm NS}_{\mathbb{R}}(X)$ we consider the \emph{ray spanned by $x$}, that is, the subset of ${\rm NS}_{\mathbb{R}}(X)$ given by
$$\mathbb{R}_+x:=\{\lambda x\mid \lambda\in \mathbb{R} \mbox{ and } \lambda>0\}.$$
Define \[
{\mathcal A}(X):=\{\text{rays }\mathbb{R}_+x\,|\, x\in {\rm NS}_{\mathbb{R}}(X)\text{ and }[H]\cdot x>0\}.\]
 Let $\mathcal{R}(X)\subseteq {\mathcal A}(X)$ be the set of rays of the form $\mathbb{R}_+[C]$, where $C$ is a negative curve on $X$. Consider the hyperplane
 $$[H]_{=1}:=\{x\in {\rm NS}_{\mathbb{R}}(X)\mid [H]\cdot x=1\}$$
endowed with the Euclidean topology. Now endow the set  $\mathcal{A}(X)$ with the finest topology making continuous the map $\tau: {\mathcal{A}}(X)\rightarrow [H]_{=1}$ given by $$x\mapsto \frac{1}{[H]\cdot x}x.$$ Let us denote by $\mathcal{R}^{ac}(X)$ the set of accumulation points of $\mathcal{R}(X)$ in $\mathcal{A}(X)$ with respect to this topology. Elements in $\mathcal{R}^{ac}(X)$ are named \emph{limit rays}.

\begin{lemma}\label{self}
If $\mathbb{R}_+x\in {\mathcal R}^{ac}(X)$ then $x^2=0$.
\end{lemma}

\begin{proof}

By hypothesis, there exists a sequence $\{C_n\}_{n=1}^\infty$ of pairwise different negative curves on $X$ such that  $$\lim_{n\rightarrow \infty}\;\frac{1}{H\cdot C_n}[C_n]=\frac{1}{[H]\cdot x}x.$$
Moreover, by \cite[Corollary 1.19(3)]{KolMor},
the sequence $\{H\cdot C_n\}_{n=1}^\infty$ is not bounded and therefore
 we can assume (considering a subsequence if necessary) that $\{H\cdot C_n\}_{n=1}^\infty$ tends to infinity.

On the one hand, since $H$ is ample and $C_n^2<0$ for all $n\in \mathbb{Z}_+$, it is clear that $x^2\leq 0$. On the other hand, by the genus formula it holds that
\begin{equation}\label{for}
\frac{1}{(H\cdot C_n)^2}+\frac{1}{2}\left(\frac{C_n^2}{(H\cdot C_n)^2}+\frac{K_X\cdot C_n}{(H\cdot C_n)^2} \right)\geq 0\quad\quad \mbox{for all $n\in \mathbb{Z}_+$}.
\end{equation}
Since the set $\{\frac{1}{H\cdot C_n}[C_n]\mid n\in \mathbb{Z}_+\}$ is contained into a bounded region of $[H]_{=1}$ \cite[Corollary 1.19(2)]{KolMor}, we deduce that the set
$\{\frac{1}{H\cdot C_n}K_X\cdot C_n\mid n\in \mathbb{Z}_+\}$ is bounded and, therefore, $(K_X\cdot C_n)/(H\cdot C_n)^2$ tends to zero. Now, considering limits when $n$ tends to infinity in Equation (\ref{for}), we conclude that $x^2\geq 0$, what concludes the proof.

\end{proof}

\begin{lemma}\label{lemma_thm_paco}
Let $X$ and $H$ be as before and let $\{U_\alpha\}_{\alpha\in \Lambda}$ be an open covering of $\mathcal{R}^{ac}(X)$. Then, there exists a positive real number $\kappa$ such that
$$\frac{D\cdot C}{H\cdot C}\geq \kappa$$
for any nef divisor $D$ on $X$ satisfying $\mathbb{R}_+[D]\not\in \bigcup_{\alpha\in \Lambda} U_\alpha$ and for any negative curve $C$ on $X$ with the property that $D\cdot C>0$.
\end{lemma}

\begin{proof}Let $\mathcal{P}$ be the set of pairs $(D,C)$ as in the statement, that is, $D$ is a nef divisor satisfying $\mathbb{R}_+[D]\not\in \bigcup_{\alpha\in \Lambda} U_\alpha$ and $C$ is a negative curve with the property that $D\cdot C>0$.

We will reason by contradiction. Therefore, assume that, for any positive real number $\epsilon$, there exists $(D,C)\in \mathcal{P}$ such that $\frac{D\cdot C}{H\cdot C}<\epsilon$. This implies the existence of a sequence $\mathcal{S}:=\{(D_n,C_n)\}_{n=1}^\infty$ of elements in $\mathcal{P}$ such that
\begin{equation}\label{eq}
\lim_{n\rightarrow \infty} \frac{D_n\cdot C_n}{H\cdot C_n}=0.
\end{equation}
Let us consider the maps $\tau_i: \mathcal{S}\rightarrow [H]_{=1}$, $i=1,2$, defined as $\tau_1(D_n,C_n):=\tau([D_n])=\frac{1}{H\cdot D_n}[D_n]$ and $\tau_2(D_n,C_n):=\tau([C_n])=\frac{1}{H\cdot C_n}[C_n]$, $\tau$ being the map defined before Lemma \ref{self}. We distinguish two cases:\medskip

- Case 1: The set $\tau_2(\mathcal{S})$ is finite. Notice that, for each $0<n\in \mathbb{Z}$, $C_n$ is the unique reduced and irreducible curve whose class belongs to the ray $\mathbb{R}_+[C_n]$. This implies that, in this case, there are finitely many curves $C_n$ involved in $\mathcal{S}$ and, therefore, there is a subsequence $\{(D_{n_\ell}, C_{n_{\ell}})\}_{\ell=1}^\infty$ of $\mathcal{S}$, a positive integer $\ell_0$ and a reduced and irreducible curve $C$ such that $C_{n_\ell}=C$ for all $\ell\geq \ell_0$. Taking (\ref{eq}) into account, this implies that $D_{n_\ell}\cdot C=0$ for all $\ell\geq \ell_0$. This is a contradiction because $(D_{n_\ell},C)\in \mathcal{P}$ for all $\ell\geq \ell_0$.\bigskip

- Case 2: The set $\tau_2(\mathcal{S})$ is infinite. Set $\overline{\rm Eff}(X)$ the closure of the effective cone $\rm Eff(X)$, since $\tau_2(\mathcal{S})\subseteq \overline{\rm Eff}(X)\cap [H]_{=1}$ (which is bounded by \cite[Corollary 1.19 (2)]{KolMor}), $\tau_2(\mathcal{S})$ has an accumulation point $c$. Hence (replacing $\{C_n\}_{n=1}^\infty$ by a suitable subsequence if necessary) we can assume, without loss of generality, that
$$\lim_{n\rightarrow \infty} \frac{1}{H\cdot C_n}[C_n]=c.$$ Moreover, since $\mathbb{R}_+c\in \mathcal{R}^{ac}(X)$, it holds that $c^2=0$ (by Lemma \ref{self}).

If the set $\tau_1(\mathcal{S})$ is infinite, then it has an accumulation point (because it is bounded, as in the case of $\tau_2(\mathcal{S})$). Otherwise $\tau_1(\mathcal{S})$ has a constant subsequence. Hence, in any case, we can assume (considering a suitable subsequence if necessary) that
the sequence $\{\frac{1}{H\cdot D_n}[D_n]\}_{n=1}^\infty$ is convergent, $d$ being its limit when $n$ tends to infinity. By Equation (\ref{eq}) one has
$$0=\lim_{n\rightarrow \infty} \frac{D_n\cdot C_n}{(H\cdot D_n)(H\cdot C_n)}=d\cdot c.$$

By the Hodge Index Theorem \cite[V, Theorem 1.9 and Remark 1.9.1]{Har}, the set $Q(X):=\{z\in {\rm NS}_{\mathbb{R}}(X)\mid z^2\geq 0 \mbox { and } [H]\cdot z\geq 0\}$ is (in a certain basis of ${\rm NS}_{\mathbb{R}}(X))$ the half-cone over an Euclidean ball $B\subseteq [H]_{=1}$ of dimension $\rho(X)-1$. Moreover, by \cite[Corollary 1.21]{KolMor}, $Q(X)\subseteq \overline{{\rm Eff}}(X)$.
Notice that $[H]\cdot c=[H]\cdot d=1$, $c$ belongs to the boundary of $B$ and $B$ is strictly convex. These facts imply that $d=c$, it is a contradiction because $\mathbb{R}_+[D_n]\not\in \cup_{\alpha\in \Lambda} U_{\alpha}$ for all $n\geq 1$ and $\mathbb
{R}_+d=\mathbb
{R}_+c\in \mathcal{R}^{ac}(X)\subseteq  \cup_{\alpha\in \Lambda} U_{\alpha}$.

\end{proof}

Now we are ready to state the aforesaid main result, which affirms that the quotient involved in the {LWBNC} is always bounded whenever the rays given by the implicated nef divisors $D$ are not near a limit ray.

\begin{theorem}\label{thm_cota_paco}
Let $X$ be a (smooth complex projective) surface and keep the above notation. Let $\{U_\alpha\}_{\alpha\in \Lambda}$ be an open covering of the set  $\mathcal{R}^{ac}(X)$. Then, there exists a positive real number $\eta$ (depending only on $X$ and the mentioned covering) which satisfies the inequality
$$\frac{C^2}{D\cdot C}\geq -\eta,$$
for any nef divisor $D$ on $X$ such that $\mathbb{R}_+[D]\not\in \bigcup_{\alpha\in \Lambda} U_\alpha$ and for any negative curve $C$ on $X$ such that $D\cdot C>0$.

\end{theorem}

\begin{proof}
By Lemma \ref{lemma_thm_paco}, there exists a positive real number $\kappa$ such that $$\frac{D\cdot C}{H\cdot C}\geq \kappa$$
for any nef divisor $D$ on $X$ such that $\mathbb{R}_+[D]\not\in \bigcup_{\alpha\in \Lambda} U_\alpha$ and for any negative curve $C$ on $X$ such that $D\cdot C>0$. Now, the result follows by applying Lemma \ref{lemma_bncota}.

\end{proof}

\begin{figure}[H]
\includegraphics[scale=0.2]{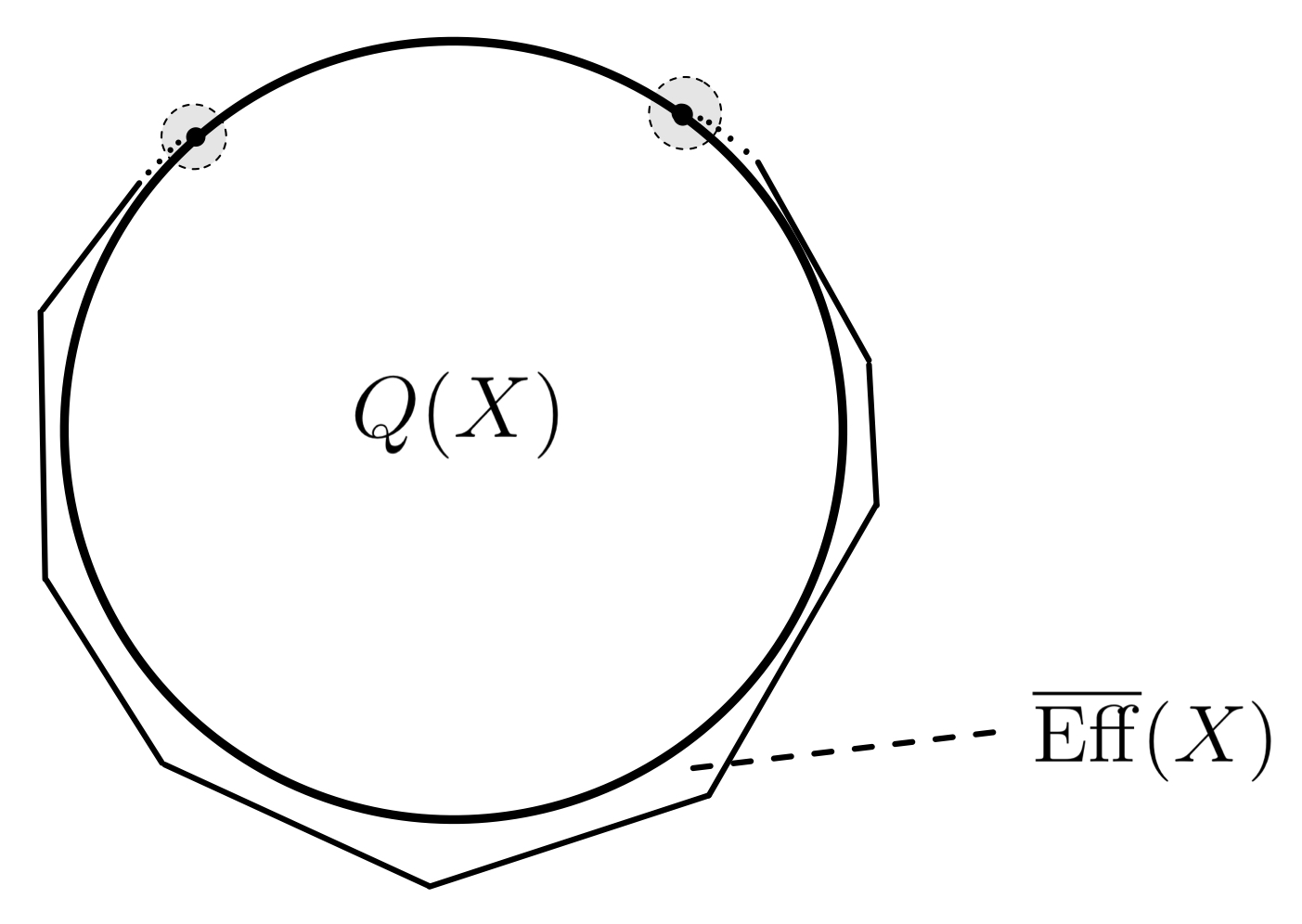}
\caption{Explanatory picture of Theorem \ref{thm_cota_paco}.}\label{figure}
\end{figure}

\begin{remark}
In Figure  \ref{figure}, we show an explanatory picture of the statement of Theorem \ref{thm_cota_paco}. Schematically one can see the section of $\overline{\rm Eff}(X)$ with the hyperplane $[H]_{=1}$ and the set $Q(X)$ introduced in the proof of Lemma \ref{lemma_thm_paco}. Small dots represent sequences of rays in $\mathcal{R}(X)$ converging to limit rays (thick dots). Finally, the shadowed circles indicate the open subsets $U_{\alpha}$ in the covering of the statement.
\end{remark}

Our last section studies linear weighted bounded negativity for rational surfaces.

\section{Explicit bounds for rational surfaces}\label{sec_explbound}

Throughout this section $S_0$ will denote, unless otherwise stated, either the complex projective plane $\mathbb{P}^2$ or the complex Hirzebruch surface $\mathbb{F}_\delta$ for some non-negative integer $\delta$.

\begin{definition}\label{def_Stuple}

An \emph{$S_0$-tuple} is a tuple $(S, S_0,\mathcal{C})$ such that $\conf=\{p_i\}_{i=1}^n$ is a configuration (of infinitely near points) over $S_0$ and $S$ is the rational surface obtained from the sequence of blowups
\[
\pi_{\mathcal{C}}:S=S_n\overset{\pi_n}{\longrightarrow}S_{n-1}\overset{\pi_{n-1}}{\longrightarrow}\cdots\overset{\pi_1}{\longrightarrow} S_0,\]
where $\pi_i$ is the blowup centered at $p_i\in S_{i-1}$ for all $i=1,\ldots,n$ (see (\ref{eqn_def_blw})).
\end{definition}

Notice that any rational surface can be regarded as the surface $S$ given by an $S_0$-tuple $(S,S_0,\conf)$.\medskip

Throughout this section we fix an $S_0$-tuple $(S,S_0,\conf)$. We will give explicit lower bounds for the values $\nu_{D}(S)$ defined in (\ref{eqn_nuD}), $D$ running over a large family of nef divisors on $S$. We include the case where $D$ is the pull-back by $\pi_{\mathcal{C}}$ of any nef divisor on $S_0$.

Firstly, we assume the existence of a foliation $\mathcal{F}$ on $S_0$ such that $\mathcal{C}\subseteq \cf$, $\cf$ being the set of infinitely near singularities of \fol. Then we provide bounds for $\nu_D(S)$ depending on \fol. Subsequently, we show how to construct a foliation as before, satisfying certain additional properties. Finally, this construction will allow us to determine explicit bounds for $\nu_{D}(S)$ which are independent of the foliation; in fact, \emph{they depend only on combinatorial data which can be extracted from the proximity graph of the configuration $\mathcal{C}$}.\medskip

The following result will be useful.

\begin{proposition}\label{prop_AG}
	Let $\mathcal{G}$ be a foliation on a rational surface $S$. Then the set of negative $\mathcal{G}$-invariant curves is finite.
\end{proposition}

\begin{proof}
Since $S$ is rational, there exists a composition of blowups centered at closed points $\pi: S\rightarrow S_0$ such that $S_0$ is either the projective plane {or} a Hirzebruch surface. Let $\mathcal{C}$ be the configuration of centers of $\pi$. Let $\pi_*\mathcal{G}$ be the foliation on $S_0$ induced by $\mathcal{G}$. As the set of exceptional (with respect to $\pi$) $\mathcal{G}$-invariant negative curves is finite, it is enough to show that there are finitely many non-exceptional negative $\mathcal{G}$-invariant curves of the form $\tilde{C}$, $C$ being a reduced and irreducible curve on $S_0$ which is $\pi_*\mathcal{G}$-invariant and $\tilde{C}$ its strict transform on $S$.

The self-intersection of such a curve $\tilde{C}$ is equal to $C^2-\sum_{p\in \mathcal{C}} {\rm m}_p^2$, where ${\rm m}_p$ denotes the multiplicity of the strict transform of $C$ at $p$. As a consequence, it suffices to assume that ${\mathcal B}_{\pi_*\mathcal{G}}\subseteq \mathcal{C}$ (blowing-up at more points of ${\mathcal B}_{\pi_*\mathcal{G}}$, if needed). Moreover, we can also assume that $\pi_*\mathcal{G}$ admits a rational first integral because, otherwise, $\pi_*\mathcal{G}$ (and also $\mathcal{G}$) has finitely many reduced and irreducible invariant curves \cite{dar}.

Let $F/G$ be a rational first integral of $\pi_*\mathcal{G}$, where $F$ and $G$ are homogeneous (respectively, bihomogeneous) polynomials of the same degree (respectively, bidegree) if $S_0=\mathbb{P}^2$ (respectively, $S_0$ is a Hirzebruch surface), and without common irreducible factors. Then, the reduced and irreducible (algebraic) $\pi_* \mathcal{G}$-invariant curves are exactly the reduced and irreducible components of the curves of the pencil $\langle F, G\rangle$.
Since $\mathcal{B}_{\pi_*\mathcal{G}}\subseteq \mathcal{C}$, the strict transforms on $S$ of two general curves of the pencil $\langle F, G \rangle$ do not meet. This implies that the strict transforms of all the curves of the pencil $\langle F, G\rangle$, except finitely many of them, have self-intersection equal to 0, what proves the result.

\end{proof}

\subsection{Bounds depending on a foliation}\label{sec_41}

In this subsection, as indicated, we fix an $S_0$-tuple $(S,S_0,\conf)$ and, moreover,  assume that \emph{there exists a foliation $\mathcal F$  on $S_0$ such that ${\mathcal C}\subseteq \cf$}.
	
Set $r$ (respectively, $(r_1,r_2)$) the degree (respectively, bidegree) of $\mathcal F$. Consider the strict transform (by $\pi_{\mathcal{C}}$) $\tilde{\mathcal{F}}$ of $\mathcal{F}$ on $S$  and define $\hat{A}(\tilde{\mathcal{F}})$ to be \emph{the set of non-exceptional (with respect to $\pi_{\mathcal{ C}}$) negative $\tilde{\mathcal{F}}$-invariant curves on $S$}. Notice that $\hat{A}(\tilde{\mathcal{F}})$ is finite by Proposition \ref{prop_AG}. Also set
\[
\hat{\alpha}(\tilde{\mathcal{F}}):=\left\{\begin{array}{ll}
	\max\{-{W^2}\mid W\in \hat{A}(\tilde{\mathcal{F}})\}&\text{if }\hat{A}(\tilde{\mathcal{F}}) \neq\emptyset,\\
	0&\text{otherwise}.
\end{array}\right.\]
Similarly, for any nef divisor $D$ on $S$, let
$$\hat{\alpha}_D(\tilde{\mathcal{F}}):=\max\left\{\frac{-W^2}{D\cdot W}\mid W\in \hat{A}(\tilde{\mathcal{F}})\; \mbox{ and }\; D\cdot W>0\right\}$$
if $\hat{A}(\tilde{\mathcal{F}})$ is not empty, and $\hat{\alpha}_D(\tilde{\mathcal{F}})=0$ otherwise.

Our first result assumes the knowledge of the foliation \fol and also that its canonical divisor $K_{\mathcal{F}}$ is nef. This fact is always  true when $S_0=\mathbb{P}^2$. For any divisor $A$ on some of the surfaces $S_i$ introduced in Definition \ref{def_Stuple}, $A^*$ will mean total transform of $A$ in $S$.

\begin{theorem}\label{thm_cota_general}
  Let $(S,S_0,\conf)$ be an $S_0$-tuple and suppose that \fol is a foliation on $S_0$ as said before. Keep the notation as in Subsections \ref{subsec_rat_surf} and \ref{subsubinv}. Consider a canonical divisor of $\mathcal{F}$, $K_{\mathcal{F}}$, and assume that it is nef. If $S_0=\mathbb{P}^2$ (respectively, $S_0=\mathbb{F}_\delta$) and $D=L^*$ (respectively, $D=F^*+M^*$), it holds that
  \begin{equation*}
\frac{C^2}{D\cdot C}\geq -\beta(\mathcal{F}),
\end{equation*}
for any negative curve $C$ on $X$ such that $D\cdot C>0$ and $C\not\in \hat{A}(\tilde{\mathcal{F}})$, where \[
\beta(\mathcal{F}):=r-1\text{ if }S_0=\mathbb{P}^2\text{ and }\beta(\mathcal{F}):=r_1+r_2\text{ otherwise}.\]

As a consequence, the value $\nu_D(S)$, introduced in (\ref{eqn_nuD}), satisfies
$$\nu_D(S)\geq \min\{-\beta(\mathcal{F}), -\hat{\alpha}_D(\tilde{\mathcal{F}})\}\geq  \min\{-\beta(\mathcal{F}), -\hat{\alpha}(\tilde{\mathcal{F}})\}.$$

\end{theorem}

\begin{proof}

Let $C$ be a negative curve on $S$ such that $C\not\in \hat{A}(\tilde{\mathcal{F}})$. Notice that $C$ is not exceptional because $D\cdot C>0$. Therefore $C$ is not $\tilde{\mathcal{F}}$-invariant.

Set ${\mathcal C}=\{p_1,\ldots,p_n\}$ and consider a canonical divisor $K_{\tilde{\mathcal F}}$  of $\tilde{\mathcal F}$. It holds that $K_{\tilde{\mathcal F}}$ is linearly equivalent to $K_{\mathcal F}^*-\sum_{i=1}^n (\nu_i({\mathcal F})+\epsilon_i-1)E_{p_i}^*$ where
$\nu_i({\mathcal F})$ is the multiplicity of the strict transform of $\mathcal F$ at $p_i$ and $\epsilon_i$ equals $1$ (respectively, $0$) if $E_{p_i}$ is not (respectively, is) $\tilde{\mathcal F}$-invariant.

Since, by hypothesis, all the points in $\mathcal{C}$ are infinitely near singularities of $\mathcal F$,  it holds that $\nu_{i}({\mathcal F})\geq 1$ for all $i$.
Then, by Lemma \ref{lemma_bncota} and the projection formula,
$$C^2\geq  -K_{\mathcal F}^*\cdot {C}+\sum_{i=1}^n (\nu_i({\mathcal F})+\epsilon_i-1){\rm mult}_{p_i}((\pi_{\mathcal{C}})_*C)\geq -K_{\mathcal F}\cdot ({\pi_{\mathcal{C}})_* C},$$
where ${\rm mult}_{p_i}((\pi_{\mathcal{C}})_*C)$ denotes the multiplicity of the strict transform of $(\pi_{\mathcal{C}})_*C$ at $p_i$.

On the one hand, if $S_0=\mathbb{P}^2$ then
${C}^2\geq -(r-1)\deg((\pi_{\mathcal{C}})_* C)=-(r-1)(L^*\cdot {C})$
and, therefore,
\begin{equation}\label{lp2}
\frac{C^2}{L^*\cdot C}\geq -(r-1).
\end{equation}
On the other hand, if $S_0=\mathbb{F}_{\delta}$,
$$
{C}^2\geq -r_2\deg_1((\pi_{\mathcal{C}})_* C)-(r_1+\delta r_2)\deg_2((\pi_{\mathcal{C}})_* C)$$
$$=-[\deg_1((\pi_{\mathcal{C}})_* C)+(1+\delta)\deg_2((\pi_{\mathcal{C}})_* C)](r_1+r_2)$$
$$+\deg_1((\pi_{\mathcal{C}})_* C) r_1+\deg_2((\pi_{\mathcal{C}})_* C) r_2+\delta \deg_2((\pi_{\mathcal{C}})_* C) r_1$$
$$=-[(F^*+M^*)\cdot C](r_1+r_2)+ \deg_1((\pi_{\mathcal{C}})_* C) r_1+\deg_2((\pi_{\mathcal{C}})_* C) r_2+\delta \deg_2((\pi_{\mathcal{C}})_* C) r_1$$
$$\geq  -[(F^*+M^*)\cdot C](r_1+r_2),
$$
where the last inequality is due to the fact that $K_{\mathcal{F}}$ is nef.

Hence,
\begin{equation}\label{lhirz}
\frac{C^2}{(F^*+M^*)\cdot C}\geq -(r_1+r_2).
\end{equation}

Finally, the last assertion of the statement follows from the first one.

\end{proof}

The following corollary is a straightforward consequence of Theorem \ref{thm_cota_general}. It shows the existence of a common bound for $\nu_D(S)$ (depending on the foliation $\mathcal{F}$) for a wide class of divisors $D$ on $S$. It uses sets $\Delta(X;G,\epsilon)$ as introduced in (\ref{eqn_Delta}).

\begin{corollary}\label{a}
Keep the notation and assumptions as in Theorem \ref{thm_cota_general}. Let $\epsilon$ be a positive real number. If $S_0=\mathbb{P}^2$ (respectively, $S_0=\mathbb{F}_\delta$) and $D$ is a divisor on $S$ such that $D\in \Delta(S;L^*,\epsilon)$ (respectively, $D\in \Delta(S;F^*+M^*,\epsilon)$), it holds that
  \begin{equation*}
\frac{C^2}{D\cdot C}\geq -\beta(\mathcal{F})/\epsilon,
\end{equation*}
for any non-exceptional negative curve $C$ on $X$ such that $D\cdot C>0$ and $C\not\in \hat{A}(\tilde{\mathcal{F}})$.

As a consequence,
$$\nu_D(S)\geq \min\{-\beta(\mathcal{F})/\epsilon, -\hat{\alpha}(\tilde{\mathcal{F}}), -\gamma \},$$
where $\gamma$ denotes the maximum of the set $\{-E_{q}^2\mid q\in \mathcal{C}\}$ if it is not empty, and $0$ otherwise.

\end{corollary}

\begin{proof}

Let $C$ be a non-exceptional negative curve on $S$ such that $D\cdot C>0$ and $C\not\in \hat{A}(\tilde{\mathcal{F}})$.

If $S_0=\mathbb{P}^2$  then $C^2/(L^*\cdot C)\geq -\beta(\mathcal{F})$ by Theorem \ref{thm_cota_general} (notice that $L^*\cdot C>0$ because $C$ is not exceptional). The first part of the statement follows, in this case, from the following formula:
$$\frac{C^2}{D\cdot C}= \frac{C^2}{(D-\epsilon L^*)\cdot C+\epsilon L^*\cdot C}\geq \frac{C^2}{\epsilon L^*\cdot C}.$$
In the case $S_0=\mathbb{F}_\delta$, the proof is similar.

The last part of the statement follows from the first one.

\end{proof}

We finish this subsection by giving another consequence of Theorem \ref{thm_cota_general}. It gives a common bound on $\nu_D(S)$ for all divisors $D$ obtained by pulling-back (by $\pi_{\mathcal{C}}$) a
nef divisor on $S_0$.

\begin{corollary}\label{b}
Let $(S,S_0,\conf)$, \fol, $\beta(\mathcal{F})$ and $\hat{\alpha}(\tilde{\mathcal{F}})$ be as in Theorem \ref{thm_cota_general}. It holds that
  \begin{equation}\label{ineq}
\frac{C^2}{D^*\cdot C}\geq -\beta(\mathcal{F}),
\end{equation}
for any nef divisor $D$ on $S_0$ and for any negative curve $C$ on $S$ such that $D^*\cdot C>0$ and $C\not\in \hat{A}(\tilde{\mathcal{F}})$.

As a consequence,
$$\nu_{D^*}(S)\geq  \min\{-\beta(\mathcal{F}), -\hat{\alpha}(\tilde{\mathcal{F}})\}$$
for all nef divisors $D$ on $S_0$.

\end{corollary}

\begin{proof}
Let $D$ be a nef divisor on $S_0$ and let $C$ be a negative curve on $S$ such that $D^*\cdot C>0$ and $C\not\in \hat{A}(\tilde{\mathcal{F}})$.\medskip

If either $S_0=\mathbb{P}^2$, or $S_0=\mathbb{F}_\delta$ and $D$ is linearly equivalent to $a F+b M$ (where $a>0$ and $b>0$), then Inequality (\ref{ineq}) follows from Corollary \ref{a} (by considering $\epsilon=1$). Thus, it is enough to assume that $S_0=\mathbb{F}_\delta$ and $D$ is either $F$ or $M$. In both cases we have
\[
\frac{C^2}{D^*\cdot C}\geq \frac{C^2}{(F^*+M^*)\cdot C}\geq -\beta(\mathcal{F}),\]
where the last inequality follows from Theorem \ref{thm_cota_general}.\medskip

The last part of the statement is straighforward from the first one (notice that $C$ cannot be not exceptional if $D^*\cdot C>0$ for some nef divisor $D$ on $S_0$).

\end{proof}

\subsection{Bounds depending only on $(S,S_0,\mathcal{C})$}\label{sec_42}

The purpose of this subsection consists of obtaining close bounds to those obtained in Section \ref{sec_41} but depending only on the $S_0$-tuple $(S,S_0,\mathcal{C})$. To this end, we divide this subsection in three parts. The second one shows the existence of a foliation on $S_0$ suitable for our purposes. The first part recalls the concept of $C^0$-sufficiency we will use in the proofs of the results of the second part. Finally we provide our bounds in the third part.

\subsubsection{$C^0$-sufficiency}

The concept of configuration introduced in Subsection \ref{subsec_conf} can be extended to infinitely many points, although unless otherwise stated, for us configuration means finite configuration. Let $\conf'=\{p'_i\}_{i\geq 1}$ be a non-necessarily finite configuration of infinitely near points over a (smooth complex projective) surface $X$ giving rise to the sequence of point blowups \[
\pi:\cdots\longrightarrow X_n\overset{\pi_n}{\longrightarrow}X_{n-1}\longrightarrow \cdots \longrightarrow X_1\overset{\pi_1}{\longrightarrow}X_0=X.\]
Set $p'_1=p$ and consider a germ (of curve) $\xi$ at $p$ on $X$. We say that $\xi$ \emph{goes through} $p'_i$ if the strict transform of $\xi$ at $p'_i$ is not empty, that is, if ${\rm mult}_{p'_i}(\xi)>0$, where ${\rm mult}_{p'_i}(\xi)$ denotes the multiplicity (of the strict transform) of $\xi$ at $p'_i$. If ${\rm mult}_{p'_i}(\xi)>1$ (respectively, ${\rm mult}_{p'_i}(\xi)=1$) we say that $p'_i$ is a \emph{multiple} (respectively, \emph{simple}) point of $\xi$.

For a germ $\xi$ as before, we denote by $\mathcal{N}(\xi)=\{p_i\}_{i\geq 1}$, $p_1=p$, the set of (equal to or infinitely near $p$) points through which $\xi$ goes. Next, we recall the concept of (infinitely near) singular point of $\xi$ (see \cite[Section 3.8]{Cas}).

Let $\xi$ and $\mathcal{N}(\xi)$ be as above. A point $p_i\in\mathcal{N}(\xi)$ is a \emph{singular point} of $\xi$ if it satisfies one of the following conditions:
\begin{enumerate}
	\item $p_i$ is a multiple point of $\xi$.
	\item $p_i$ is a satellite point of $\mathcal{N}(\xi)$.
	\item $p_i$ precedes a satellite point of $\mathcal{N}(\xi)$.
\end{enumerate}
Otherwise, $p$ is a \emph{non-singular point} of $\xi$.

Let $\xi$ be a non-empty reduced germ defined on a surface $X$ and set $\xi_1,\ldots,\xi_s$ the branches of $\xi$. For $i=1,\ldots,s$, denote by $q_i$ the first point in $\mathcal{N}(\xi_i)$ which is a non-singular point of $\xi$. This implies that all points in $\mathcal{N}(\xi_i)$ infinitely near $q_i$ are simple and free.

Keep the notation as in Subsection \ref{subsec_conf}. The (finite) \emph{configuration of $\xi$} is defined as
\begin{equation}\label{eqn_conf_xi}
	\conf(\xi)=(\conf)^{q_1}\cup\cdots\cup(\conf)^{q_s}
\end{equation}
and the \emph{singular configuration} of $\xi$, $\mathcal{K}(\xi)$, is the subset of $\conf(\xi)$ made up of all their singular points, i.e., \[
\mathcal{K}(\xi)=\conf(\xi)\setminus \mathcal{E}_{\conf(\xi)},\]
where $\mathcal{E}_{\conf(\xi)}$ is the set of ends of $\conf(\xi)$.

Let $\conf=\{p_i\}_{i=1}^n=\cup_{q\in \mathcal{E}_\conf}(\conf)^q$ be a configuration. We define the \emph{vector of multiplicities of $\conf$} as the column vector $\mathbf{m}_{\conf}=(m_{\conf,1},\ldots,m_{\conf,n})^t$, where \begin{equation}\label{eqn_vect_mult}
	m_{\conf,i}:=\left\lbrace \begin{array}{ll}
		1&\text{if }p_i\in \mathcal{E}_\conf,\\
		\sum_{p_j\rightarrow p_i}m_{\conf,j}&\text{otherwise}.
	\end{array}\right.
\end{equation}

Let us denote by $\mathcal{O}_{p}\cong\C^2$ the local ring of the surface $X$ at $p$ and set $\mathfrak{m}_p$ its maximal ideal. Let $\xi:f=0$ be a non-empty and reduced germ of curve at $p$ (in $\mathcal{O}_{p}$). A positive integer $d$ is said to be \emph{$C^0$-sufficient} for $\xi$ if all the elements of the form $f+\mathfrak{m}_p^d$ are non-zero and define an equisingular to $\xi$ reduced germ (see \cite[Section 3.8]{Cas} for the definition of \emph{equisingularity}). It is clear that this definition does not depend on a particular equation $f$ of $\xi$. If $f=\sum_{i+j=0}^{\infty} f_{ij}x^iy^j$, where $\{x,y\}$ is a regular system of parameters of $\mathfrak{m}_p$, then the polynomial $f_0=\sum_{i+j=0}^{d-1} f_{ij}x^iy^j$ defines an equisingular to $\xi$ germ. In addition, the fact that $d$ is $C^0$-sufficient for $\xi$ implies that any $d'>d$ is also $C^0$-sufficient for $\xi$. We define the \emph{$C^0$-sufficiency degree} of a reduced germ $\xi$ as the minimum $d$ such that $d$ is $C^0$-sufficient for $\xi$.

The next result provides an upper bound for the $C^0$-sufficiency degree of a non-empty germ of curve $\xi$ at a point $p$ on a surface $X$. It is a consequence of \cite[Theorem 7.5.1]{Cas}.

\begin{lemma}\label{lemma_c0_suff}
	Let $X$ be a surface, $p$ a point in $X$ and $\xi:f=0$ a reduced germ of curve at $p$ on $X$. Set $\mathcal{K}=\mathcal{K}(\xi)=\{p_i\}_{i=1}^n$, $p_1=p$, its singular configuration. Let $\mathbf{P}_\mathcal{K}$ be the proximity matrix of $\mathcal{K}$ (see (\ref{eqn_prox_mtx})) and $\mathbf{m}_{\mathcal{K}}$ its vector of multiplicities. Let $\mathbf{v}_d=(v_i)^t$ be the vector defined by $\mathbf{v}_d=\mathbf{P}^{-1}_\mathcal{K}(d\, \mathbf{1}_p-\mathbf{m}_\mathcal{K})$, where $\mathbf{1}_p$ is the $n$-dimensional column vector whose first component is $1$ and the remaining ones are $0$. Then,
	\begin{enumerate}
		\item The least positive integer $d$ such that $v_i>0$ for all $i\in\{1,\ldots,n\}$ is $C^0$-sufficient for $\xi$.
		\item For any $g\in\mathfrak{m}_p^d$, $f+g\neq 0$ and the germ of curves $\zeta:f+g=0$ goes through $\mathcal{K}$. Moreover the vector of multiplicities of $\conf(\zeta)$ (respectively, its singular configuration) satisfies $\mathbf{m}_{\conf(\zeta)}=\mathbf{m}_{\mathcal{K}}$ (respectively, $\mathcal{K}(\zeta)=\mathcal{K}$).
	\end{enumerate}
\end{lemma}

Notice that the integer $d$ introduced in Lemma \ref{lemma_c0_suff} needs not to be the $C_0$-sufficiency degree of $\xi$.

\subsubsection{Algebraically integrable foliations through a configuration on $S_0$}

This subsection has to do with algebraically integrable foliations on surfaces $S_0$ as those introduced at the beginning of Section \ref{sec_explbound}. We keep the notation as in Subsection \ref{subsec_foliation} and prove the existence of a foliation \fol on $S_0$ admiting a rational first integral and such that the set $\Bf$ contains a prefixed configuration over $S_0$.

The following lemma, which generalizes \cite[Lemma 3.4]{GalMonMorPerC2022}, will be useful. For any polynomial $f(x,y)\in\C[x,y]$, we denote by $\deg_x(f)$ (respectively, $\deg_y(f)$) the degree of $f$ regarded as a polynomial on $x$ (respectively, $y$), i.e., $f\in \C[y][x]$ (respectively, $f\in \C[x][y]$). We use the notation introduced in Subsection \ref{subsec_rat_surf}.

\begin{lemma}\label{lemma_grado_extension}
Let $p$ be a point in $S_0:=\PP^2$ (respectively, $S_0:=\fd$, for some $\delta\in\Z_{\geq 0}$). Consider an open affine subset $U\in\{U_X,\,U_Y,\,U_Z\}$ (respectively, $U\in\{U_{00},U_{01},U_{10},U_{11}\}$) such that $p\in U$ and take affine coordinates $(x,y)$ in $U$. Let $f(x,y)=0$ be the equation of a curve $B$ on $U$, passing through $p$, where $f(x,y)=\sum_{i+j=0}^d f_{ij}x^iy^j\in\C[x,y]$ is a polynomial of total degree $d$. Then, the closure of $B$ in $\PP^2$ (respectively, $\fd$), denoted by $D$, is 
 linearly equivalent to $d(f)L$ (respectively, $d_1(\delta,f)F+d_2(\delta,f)M$), where
\begin{gather*}
d(f)=\deg(f),
\end{gather*}
(respectively, $d_1(\delta,f)\leq \deg_x(f)\leq \deg(f)$ and $d_2(\delta,f)=\deg_y(f)\leq \deg(f)$).
Moreover, assuming $S_0=\fd$ and $f_{d0}\cdot f_{0d}\neq 0$, the following equalities hold:
\begin{enumerate}
\item If $U=U_{00}$ or $U=U_{10}$ or $\delta=0$, \[
d_1(\delta,f)=\deg(f)\text{ and }d_2(\delta,f)=\deg(f).\]
\item Otherwise (i.e., $\delta\neq 0$ and $U=U_{01}$ or $U=U_{11}$), \[
d_1(\delta,f)=0\text{ and }d_2(\delta,f)=\deg(f).\]
\end{enumerate}
\end{lemma}

\begin{proof}

We start with the case where $p=(p_1:p_2:p_3)\in\PP^2$. Let $G(X,Y,Z)=0$ be a homogeneous equation of $D$.

 Assume, without loss of generality, that $p=(0:0:1)\in U_Z$ (the other cases work similarly). Then, \[
 G(X,Y,Z)=Z^{d}f\left(x,y\right),\]
  where $x:=\frac{X}{Z}$ and $y:=\frac{Y}{Z}$ and it is clear that $\deg(G)=d$.

\smallskip

Suppose now that $p\in\fd$, for $\delta\geq 0$. Let $G(X_0,X_1,Y_0,Y_1)$ be a bihomogeneous polynomial such that $G(X_0,X_1,Y_0,Y_1)=0$ is an equation of $D$. Consider the set $\mathcal{M}$ of monomials $X_0^{a_0}X_1^{a_1}Y_0^{b_0}Y_1^{b_1}$ appearing in the expression of $G(X_0,X_1,Y_0,Y_1)$ with non-zero coefficient. We divide our study in two cases: \medskip

(1) $U=U_{00}$ (respectively, $U=U_{10}$). Here $p= (1,a;1,b)$ (respectively, $p= (a,1;1,b)$) for some $a,b\in\C$ and
 \begin{gather*}
 G(X_0,X_1,Y_0,Y_1)=X_0^{{s}_1}Y_0^{{s}_2}f(x,y)\\
 \bigl(\text{respectively, } G(X_0,X_1,Y_0,Y_1)=X_1^{{s}_1}Y_0^{{s}_2}f(x,y)\bigl),
 \end{gather*}
 where $x:=\frac{X_1}{X_0}$ and $y:=\frac{X_0^\delta Y_1}{Y_0}$ (respectively, $x:=\frac{X_0}{X_1}$ and $y:=\frac{X_1^\delta Y_1}{Y_0}$) are affine coordinates in $U_{00}$ (respectively, $U_{10}$) and ${s}_1,{s}_2\in\Z_{\geq 0}$. We can assume, without loss of generality (performing a suitable change of variables if necessary), that $p=(1,0;1,0)$ (respectively, $p=(0,1;1,0)$) and, therefore, the affine coordinates of $p$ on $U$ are $(0,0)$. Then,
 \begin{gather*}
 f(x,y)=\sum_{i+j=0}^d f_{ij}x^iy^j=\sum_{i+j=0}^d f_{ij}\left(\frac{X_1}{X_0}\right)^i\left(\frac{X_0^\delta Y_1}{Y_0}\right)^j\\
\Biggl( \text{respectively, }  f(x,y)=\sum_{i+j=0}^d f_{ij}x^iy^j=\sum_{i+j=0}^d f_{ij}\left(\frac{X_0}{X_1}\right)^i\left(\frac{X_1^\delta Y_1}{Y_0}\right)^j\Biggl).
 \end{gather*}
 Since neither $X_0$ (respectively, $X_1$) nor $Y_0$ divide $G$, there exists a monomial $X_0^{{a}_0}X_1^{{a}_1}Y_0^{{b}_0}Y_1^{{b}_1}$ in $\mathcal{M}$ with ${a}_0=0$ (respectively, ${a}_1=0$) and another one with ${b}_0=0$. Hence, as ${a}_0={s}_1-i+\delta j$ (respectively, ${a}_1={s}_1-i+\delta j$) and ${b}_0={s}_2-j$ in the monomial in $\mathcal{M}$ coming from the monomial of $f(x,y)$ with coefficient $f_{ij}$, it holds that ${s}_1\leq \deg_x(f)$ and ${s}_2=\deg_y(f))$. Moreover, if $f_{d0}\cdot f_{0d}\neq 0$, ${s}_1= \deg_x(f)=d,\,{s}_2=d$ and then, $(d_1(\delta,f),d_2(\delta,f))=(d,d)$. \medskip

(2) $U=U_{01}$ (respectively, $U=U_{11}$). Then $p= (1,a;b,1)$ (respectively, $p= (a,1;b,1)$) for some $a,b\in\C$ and
 \begin{gather*}
 G(X_0,X_1,Y_0,Y_1)=X_0^{{s}_1}Y_1^{{s}_2}f(x,y)\\
 \bigl(\text{respectively, } G(X_0,X_1,Y_0,Y_1)=X_1^{{s}_1}Y_1^{{s}_2}f(x,y)\bigl),
 \end{gather*}
 where $x:=\frac{X_1}{X_0}$ and $y:=\frac{Y_0}{X_0^\delta Y_1}$ (respectively, $x:=\frac{X_0}{X_1}$ and $y:=\frac{Y_0}{X_1^\delta Y_1}$)  are affine coordinates in $U_{01}$ (respectively, $U_{11}$) and ${s}_1,{s}_2\in\Z_{\geq 0}$. 
 Very similar arguments to the proof of (1) show the existence of a monomial $X_0^{{a}_0}X_1^{{a}_1}Y_0^{{b}_0}Y_1^{{b}_1}$ in $\mathcal{M}$ with ${a}_0=0$ (respectively, ${a}_1=0$) and another one with ${b}_1=0$. Hence, 
  it holds that ${s}_1\leq \deg_x(f)+\delta \deg_y(f)$ and ${s}_2=\deg_y(f)$. Moreover, if $f_{d0}\cdot f_{0d}\neq 0$, ${s}_1= \max\{d,\delta d\},\,{s}_2=d$ and therefore, if $\delta=0$ (respectively, $\delta\neq 0$) $(d_1(\delta,f),d_2(\delta,f))=(d,d)$ (respectively, $(d_1(\delta,f),d_2(\delta,f))=(0,d)$). \medskip

In both cases, $d_1(\delta,f)\leq\deg_x(f)\leq \deg(f)$ and $d_2(\delta,f)=\deg_y(f)\leq\deg(f)$, which ends the proof.

\end{proof}

To prove our main result in this subsection (Theorem \ref{thm_attch_fol}) we need a second lemma. Before stating and proving it, we fix some notation.\medskip

Let ${\mathcal C}_p$ be a configuration over an open subset $U$ of $S_0$ with a unique proper point $p$ (that is, $O_{{\mathcal C}_p}=\{p\}$). Let $\mathcal{E}_{{\mathcal C}_p}=\{q_1,\ldots, q_s\}$ be the set of ends of the configuration ${\mathcal C}_p$ and let $W$ be the subset of free points in $\mathcal{E}_{{\mathcal C}_p}$.

If $\# {\mathcal C}_p\geq 2$, for any $q_j\in W$, set $q'_j$ the only satellite point in the exceptional divisor given by blowing-up at $q_j$ and consider the configuration
$$\hat{\mathcal C}_p:=\left(\bigcup_{q_j\not\in W} ({\mathcal C})^{q_j}\right)\cup \left(\bigcup_{q_j\in W} \Bigl(({\mathcal C})^{q_j}\cup \{q_j'\}\Bigl)\right).$$
Otherwise (i.e., when ${\mathcal C}_p=\{p\}$), let $\hat{\conf}_p:={\mathcal C}_p$.

 Set $\hat{\mathcal C}_p=\{p_1=p,\ldots,p_n\}$ and, attached to ${\mathcal C}_p$, let us define the following positive integer:
\begin{equation}\label{positiveinteger}
d_{{\mathcal C}_p}:=\min \{d\in \mathbb{Z}_{>0}\mid {\mathbf P}_{\hat{\mathcal C}_p}^{-1}(d {\mathbf 1}_{\hat{\mathcal C}_p}-\mathbf{m}_{\hat{\mathcal C}_p})>{\mathbf 0}\},
\end{equation}
where $\mathbf{P}_{\hat{\conf}_p}$ (respectively, $\mathbf{m}_{\hat{\conf}_p}$) is the proximity matrix (\ref{eqn_prox_mtx}) (respectively, the vector of multiplicities (\ref{eqn_vect_mult})) of $\hat{\conf}_p$ and $\mathbf{1}_{\hat{\conf}_p}$ the $\#\hat{\conf}_p$-dimensional column vector whose first coordinate is $1$ and any other coordinate is $0$. Note that we say that a vector $(v_1,\ldots,v_n)^t>{\mathbf 0}$ whenever $v_i>0$ for all $1\leq i\leq n$. In addition, $d_{{\mathcal C}_p}=2$ if ${\mathcal{C}}=\{p\}$.


\begin{lemma}\label{lemanuevo}
Let $p$ be a point in $S_0:=\mathbb{P}^2$ (respectively, $S_0:=\mathbb{F}_{\delta}$). Consider an affine open subset $U\in \{ U_X,\, U_Y,\, U_Z\}$ (respectively, $U\in \{U_{00},\,U_{01},\, U_{10},$ $U_{11}\}$) as introduced in Subsection \ref{subsec_rat_surf} such that $p\in U$ and take the before presented affine coordinates $(x,y)$ in $U$. Let $Q\subseteq S_0$ be a (possibly empty) finite set such that $p\not \in Q$ and let ${\mathcal C}_p$ be a configuration over $U\cong \mathbb{C}^2$ such that $O_{{\mathcal C}_p}=\{p\}$. Consider the positive integer $d_{{\mathcal C}_p}$ given in (\ref{positiveinteger}).

Then, there exist a polynomial $f\in \mathbb{C}[x,y]$ of degree less than or equal to $d_{{\mathcal C}_p}-1$ and a polynomial $g(x,y):=\lambda_1 x^{d_{{\mathcal C}_p}}+\lambda_2 y^{d_{{\mathcal C}_p}}+s(x,y)$, with $\lambda_1,\lambda_2\in \mathbb{C}\setminus \{0\}$, $s(x,y)\in \mathbb{C}[x,y]$ and $\deg(s)<d_{{\mathcal C}_p}$, satisfying the following conditions:
\begin{itemize}
\item[(a)] $f(p)=g(p)=0$.
\item[(b)] No point $q\in Q$ belongs to the closure in $S_0$ of the affine curve in $U$ with equation $g(x,y)=0$.
\item[(c)] $f$ and $g$ have no non-constant common factors.
\item[(d)] The pencil ${\mathcal P}$ of affine plane curves given by the equations $\alpha f(x,y)+\beta g(x,y)=0$ (where $(\alpha:\beta)$ runs over $\mathbb{P}^1$) satisfies that ${{\mathcal C}_p}\subseteq BP({\mathcal P})$.

\end{itemize}
\end{lemma}

\begin{proof}

Set $\hat{\mathcal C}_p=\{p_1=p,\ldots,p_n\}$ as introduced before the statement. Suppose that the coordinates of $p$ in $U$ are $(a,b)$ and let us define $x':=x-a$ and $y':=y-b$ in such a way that $p$, in the affine coordinates $(x',y')$, becomes the origin. For each $t\in \mathcal{E}_{\hat{\mathcal C}_p}$, let $\xi_t$ be an analytically irreducible germ of curve at $p$ such that its strict transform $\tilde{\xi}_t$ on the surface containing the exceptional divisor $E_t$ is smooth and transversal to $E_t$ at a general point. Also, identify $(x',y')$ with their images in the local ring of analytic germs ${\mathcal O}_{\mathbb{C}^2,p}=\mathbb{C}\{x',y'\}$, and pick a convergent power series $h_t(x',y')\in \mathbb{C}\{x',y'\}$ defining $\xi_t$. Let us consider the germ $\xi$ at $p$ defined by the power series $$\prod_{t\in \mathcal{E}_{\hat{\mathcal C}_p}} h_t(x',y')=\sum_{i+j=1}^\infty c_{ij}(x')^i (y')^j,\;\; c_{ij}\in \mathbb{C}\;\mbox{for all } i,j.$$

By (1) of Lemma \ref{lemma_c0_suff}, the positive integer $d_{{\mathcal C}_p}$ defined in (\ref{positiveinteger}) is $C^0$-sufficient for $\xi$. Consider the polynomial $f'(x',y'):=\sum_{i+j=1}^{d_{{\mathcal C}_p}-1} c_{ij}(x')^i (y')^j\in \mathbb{C}[x',y']$. Fix two general non-zero complex numbers $\lambda_1$ and $\lambda_2$
and let us define
$g'(x',y'):=\lambda_1(x')^{d_{{\mathcal C}_p}}+\lambda_2 (y')^{d_{{\mathcal C}_p}}$.
Notice that $f'(x',y')$ and $\lambda_1(x')^{d_{{\mathcal C}_p}}+\lambda_2 (y')^{d_{{\mathcal C}_p}}$ do not have non-constant common factors (because $\lambda_1$ and $\lambda_2$ are chosen to be general).

Let $\eta$ be the germ at $(0,0)$ of a general curve of the pencil of affine curves with equations $\alpha f'(x',y')+\beta g'(x',y')=0$, where $(\alpha:\beta)\in \PP^1$. Let $\mathcal{K}(\eta)$ be the singular configuration of $\eta$. By definition, $d_{{\mathcal C}_p}$ is $C^0$-sufficient for $\eta$. Moreover, by Lemma \ref{lemma_c0_suff} (2), ${\mathcal K}(\eta)=\mathcal{K}(\xi)=\hat{\mathcal C}_p$ and, therefore, ${{\mathcal C}_p}\subseteq {\mathcal K}(\eta)$. Notice that ${\mathcal K}(\eta)$ is contained into the configuration of base points of the above pencil by Bertini's Theorem (see \cite[Chapter III, Corollary 10.9]{Har}).

 The result follows by considering the pencil ${\mathcal P}$ of affine curves defined by the equations
$$\alpha f(x,y)+\beta g(x,y)=0,\;\;(\alpha:\beta)\in \mathbb{P}^1,$$
where $f(x,y):=f'(x-a,y-b)$ and $g(x,y):=g'(x-a,y-b)$, and by noticing that Conditions (a), (c) and (d) of the statement are satisfied by the construction of $\mathcal P$, and that,  if $\lambda_1,\lambda_2$ are chosen to be general enough, Condition (b) holds as well.

\end{proof}

To conclude this subsection, we state and prove our main result concerning $S_0$-tuples as introduced in Definition \ref{def_Stuple} and algebraically integrable foliations on $S_0$. It will be useful in our forthcoming bounds related to the LWBNC.

\begin{theorem}\label{thm_attch_fol}
Consider an $S_0$-tuple $(S,S_0,\conf)$. Then, there exists an algebraically integrable foliation $\mathcal F$ on $S_0$ such that ${\mathcal C}\subseteq {\mathcal B}_{\mathcal F}$ and the degree $r$ (respectively, bidegree $(r_1,r_2)$) of $\mathcal F$ is bounded as follows:
$$r\leq 2d-2\;\;\mbox{(respectively, $r_1\leq 2d+\delta-2$ and $r_2\leq 2d-2$),}$$
where $d:=\sum_{p\in O_{\mathcal C}}d_{({\mathcal C})_p}$, $d_{({\mathcal C})_p}$ being as defined in (\ref{positiveinteger}).

Moreover, there exists a rational first integral of $\mathcal{F}$ of degree (respectively, bidegree) $d$ (respectively, $({d_1},{d_2})$ with ${d_1}\leq d,\,{d_2}=d$).

\end{theorem}

\begin{proof}

Assume that $S_0=\mathbb{P}^2$ (respectively, $S_0=\mathbb{F}_{\delta}$). For every point $p\in O_{\mathcal C}$ take an open subset $U_p\in \{U_X,U_Y,U_Z\}$ (respectively, $U_p\in \{U_{00},U_{01},U_{10}, U_{11}\}$) such that $p\in U_p$. Consider the configuration $(\conf)_p$ over $U_p$ and the set $Q=O_\conf\setminus\{p\}$, and then apply Lemma \ref{lemanuevo} to get an irreducible pencil ${\mathcal P}_p$ of affine curves on $U_p$. It satisfies
\begin{equation}\label{eqn_CBP}
({\mathcal C})_p\subseteq BP({\mathcal P}_p).
\end{equation}
 Let $F_p(X,Y,Z)=0$ and $G_p(X,Y,Z)=0$ (respectively, $F_p(X_0,X_1,Y_0,Y_1)=0$ and $G_p(X_0,X_1,Y_0,Y_1)=0$) be the equations of the closures on $\mathbb{P}^2$ (respectively, $\mathbb{F}_{\delta}$) of two general enough curves of the pencil ${\mathcal P}_p$. Then the polynomials \[
 F:=\prod_{p\in O_{\mathcal C}} F_p\;\mbox{ and }\;G:=\prod_{p\in O_{\mathcal C}} G_p\]
have no non-constant common factors. If $S_0=\mathbb{P}^2$, then it is clear that $F$ and $G$ are polynomials of degree $d$ and, if $S_0=\mathbb{F}_{\delta}$, as a consequence of Lemma \ref{lemma_grado_extension}, $F$ and $G$ are polynomials of the same bidegree $({d_1},{d_2})$ such that ${d_1}\leq d$ and ${d_2}=d$. Notice that Lemma \ref{lemanuevo} shows that, for all $p\in O_{\mathcal C}$, the polynomials in $\mathbb{C}[x,y]$ defining the restrictions to $U_p$ of the curves with equations $F_p=0$ and $G_p=0$ have monomials $x^{d_{{(\mathcal C)_p}}}$ and $y^{d_{{(\mathcal C)_p}}}$ with non-zero coefficients. Therefore we can consider the irreducible pencil ${\mathcal P}_{S_0}$ of curves on $S_0$ defined by the equations $\alpha F+\beta G=0$, where $(\alpha:\beta)$ runs over $\mathbb{P}^1$.

Condition (b) of Lemma \ref{lemanuevo} guarantees that, for all $p\in O_{\mathcal C}$, the germs at $p$ of the curves in ${\mathcal P}_{S_0}$ coincide with those of the curves in ${\mathcal P}_p$. Therefore \[
\bigcup_{p\in O_{\mathcal C}} BP({\mathcal P}_p)\subseteq BP({\mathcal P}_{S_0})\]
and, since ${\mathcal C}=\cup_{p\in O_{\mathcal C}} ({\mathcal C})_p$, by (\ref{eqn_CBP}), one has that ${\mathcal C}\subseteq BP({\mathcal P}_{S_0})$.

The homogeneous (or bihomogeneous) 1-form $FdG-GdF$ can be factorized as $FdG-GdF=H \Omega$, where $\Omega$ is a reduced homogeneous (or bihomogeneous) 1-form and $H$ is a homogeneous (or bihomogeneous) polynomial. Let ${\mathcal F}$ be the foliation on $S_0$ defined by $\Omega$. Notice that $\pf={\mathcal P}_{S_0}$ and hence, $F/G$ is a rational first integral of $\mathcal{F}$ of degree $d$ (respectively, bidegree $({d_1},d)$, with ${d_1}\leq d$) if $S_0=\PP^2$ (respectively, $S_0=\fd$). Moreover, since ${\mathcal B}_{\mathcal F}=BP(\pf)$ {(by Proposition \ref{prop_bfbp})}, ${\mathcal C}\subseteq {\mathcal B}_{\mathcal F}$ holds.

If $S_0=\mathbb{P}^2$ (respectively, $S_0=\mathbb{F}_{\delta}$) and
\begin{gather*}
\Omega=AdX+BdY+CdZ;\;A,B,C\in\C[X,Y,Z],\\
\text{(respectively, }\Omega=A_{\delta,0}dX_0+A_{\delta,1}dX_1+B_{\delta,0}dY_0+B_{\delta,1}dY_1;\\
A_{\delta,0},\,A_{\delta,1},\,B_{\delta,0},\,B_{\delta,1}\in\C[X_0,X_1,Y_0,Y_1]\text{)},
\end{gather*}
the degree (respectively, bidegree) of ${\mathcal F}$ is $\deg(A)-1$ \[
\text{(respectively, }(\deg_1(A_{\delta,0})+\delta-1,\deg_2(A_{\delta,0})-2)\text{)},\]
where $\deg(f)$ (respectively, $(\deg_1(f),\deg_2(f))$) stands for the degree (respectively, bidegree) of a polynomial $f$ as defined in Subsection \ref{subsec_rat_surf}. As $\Omega=\frac{FdG-GdF}{H}$ and the degrees (respectively bidegrees) of $F$ and $G$ are equal to $d$ (respectively, $({d_1},{d_2})$), one gets that the degree (respectively, bidegree) of ${\mathcal F}$ is, at most, $2d-2$ (respectively, $(r_1,r_2)$ such that $r_1\leq 2d+\delta-2$ and $r_2\leq 2d-2$). This finishes the proof.

\end{proof}

\subsubsection{The bounds}

We conclude this subsection by giving explicit bounds for quotients involved in the LWBNC of any rational surface $S$.

Let $(S,S_0,\conf)$ be an $S_0$-tuple (Definition \ref{def_Stuple}) and consider \emph{an attached to $(S,S_0,\conf)$ foliation,} which is any foliation satisfying the conditions given in the statement of Theorem \ref{thm_attch_fol}. Notice that Theorem \ref{thm_attch_fol} guarantees its existence.\medskip

We will use the next lemma for proving the main result in this subsection. Keep the notation as in Subsections \ref{subsec_conf} and \ref{subsec_rat_surf}. Recall that, for any divisor $A$ on a surface $S_i$ as in Definition \ref{def_Stuple}, $A^*$ means its total transform on $S$.

\begin{lemma}\label{lem_bnd_deg}
Let $S_0=\mathbb{F}_\delta$ and set $\conf=\{p_i\}_{i=1}^n$. Suppose that the strict transform $\tilde{C}$ on $S$ of a curve $C$ on $\fd$ is linearly equivalent to the divisor ${a} F^*+{b} M^*-\sum_{i=1}^n {m}_{p_i}E_i^*$. Then, \[
{m}_{p_i}\leq{a}+{b}+\delta {b},\;\text{ for all }1\leq i \leq n.\]
\end{lemma}

\begin{proof}
 Let \[
F(X_0,X_1,Y_0,Y_1)=\sum f_{a_0,a_1,b_0,b_1}X_0^{a_0}X_1^{a_1}Y_0^{b_0}Y_1^{b_1}=0\]
be a homogeneous equation of $C$ and consider the set $\mathcal{M}$ of monomials $X_0^{a_0}X_1^{a_1}Y_0^{b_0}Y_1^{b_1}$ appearing in the expression of $F$ such that $f_{a_0,a_1,b_0,b_1}\neq 0$. Notice that
\begin{equation}\label{aux_gr}
a_0+a_1-\delta b_1={a}\; \text{ and }\;b_0+b_1={b}
\end{equation}
for all monomials in $\mathcal{M}$. Write $\conf=\cup_{p\in O_\conf}(\conf)_p$. In addition, if $q\in(\conf)_p$, then ${m}_{q}\leq {m}_p$.

To conclude the proof it suffices to show that ${m}_p\leq {a}+{b}+\delta {b}$ for all $p\in O_\conf$. Consider an affine open subset $U_{jk}\in\{U_{00},U_{01},U_{10},U_{11}\}$ such that $p\in U_{jk}$ and take the above considered affine coordinates $(x,y)\in U_{jk}$. It holds that $\sum f_{a_0,a_1,b_0,b_1} x^{\bar{a}_j}y^{\bar{b}_k}=0$, where $\bar{a}_j=\frac{a_0a_1}{a_j}$ and $\bar{b}_k=\frac{b_0b_1}{b_k}$, is an equation of $C$ on $U_{jk}$. Hence, \[
{m}_p\leq \bar{a}_j+\bar{b}_k \leq a_0+a_1 -\delta b_1 +\delta b_1+\delta b_0 +b_0+b_1\leq {a} + {b} + \delta{b},\]
where the last inequality is a consequence of (\ref{aux_gr}).
\end{proof}

The following result together with the forthcoming Corollaries \ref{cor_expl1} and \ref{cor_cotaejemplo} are the main results in Subsection \ref{sec_42}. They provide the before announced explicit bounds related to linear weighted bounded negativity. Keep the notation as in Subsection \ref{subsec_rat_surf} and recall that ${}^*$ means total transform.

\begin{theorem}\label{teo2}
Let $(S,S_0,\mathcal{C})$ be an $S_0$-tuple and \fol a foliation attached to this $S_0$-tuple. Denote by $\tilde{\mathcal{F}}$ the strict transform of \fol on $S$. Let $C$ be a reduced and irreducible curve on $S$ such that $L^*\cdot C>0$ (respectively, $(F^*+M^*)\cdot C>0$).

\begin{itemize}

\item[(a)] If $C$ is not $\tilde{\mathcal{F}}$-invariant, then
\begin{gather*}
\frac{C^2}{L^*\cdot C}\geq 3-2d\text{ if }S_0=\PP^2,\text{ and }\\
\frac{C^2}{\left(F^*+M^*\right)\cdot C}\geq 2-2d-\delta\text{ if }S_0=\fd,
\end{gather*}
where $d=\sum_{p\in O_\conf} d_{(\conf)_p}$, $d_{(\conf)_p}$ being the integer defined in (\ref{positiveinteger}).

\item[(b)] If $C$ is $\tilde{\mathcal{F}}$-invariant, then \begin{gather*}
\frac{C^2}{L^*\cdot C}\geq d(1-n)\text{ if }S_0=\PP^2,\text{ and}\\
\frac{C^2}{(F^*+M^*)\cdot C}\geq \min\{-n-\delta,-(\delta+2)dn\}\text{ if }S_0=\fd,
\end{gather*}
where $n=\#\conf$ and $d=\sum_{p\in O_\conf} d_{(\conf)_p}$, $d_{(\conf)_p}$ being the integer defined in (\ref{positiveinteger}).

\end{itemize}

\end{theorem}

\begin{proof}
We start by proving Item (a). Assume that $C$ is not $\tilde{\mathcal{F}}$-invariant. Notice that $C$ is not exceptional (because $L$ and $F+M$ are ample divisors).

On the one hand, if $S_0=\mathbb{P}^2$, applying Theorems \ref{thm_cota_general} and \ref{thm_attch_fol}, we have that
$$\frac{C^2}{L^*\cdot C}\geq -(r-1)\geq 3-2d,$$
where $r$ is the degree of the foliation $\mathcal{F}$.\medskip

On the other hand, if $S_0=\fd$, as showed in the proof of Theorem \ref{thm_cota_general}, it holds that
  \[
C^2\geq -r_2\deg_1((\pi_{\mathcal{C}})_*C)-(r_1+\delta {r_2})\deg_2((\pi_{\mathcal{C}})_*C),\]
where $(r_1,r_2)$ denotes the bidegree of $\mathcal{F}$ and $(\deg_1((\pi_{\mathcal{C}})_*C),\deg_2((\pi_{\mathcal{C}})_*C))$ the bidegree of $(\pi_{\mathcal{C}})_*C$. By Theorem \ref{thm_attch_fol},
\[
r_1\leq 2d-2+\delta,\qquad r_2\leq 2d-2.\]
Then,
\begin{align*}
C^2&\geq (2-2d)\deg_1((\pi_{\mathcal{C}})_*C)-(2d-2-\delta+ 2d\delta)\deg_2((\pi_{\mathcal{C}})_*C)\\
&= (2-2d)\left[\deg_1((\pi_{\mathcal{C}})_*C)+\delta\deg_2((\pi_{\mathcal{C}})_*C)\right]-(2d-2+\delta)\deg_2((\pi_{\mathcal{C}})_*C)\\
&= (2-2d)\left(M^*\cdot C\right)-(2d-2+\delta)(F^*\cdot C)\\
&= (2-2d)\left((F^*+M^*)\cdot C\right)-\delta(F^*\cdot C),
\end{align*}
which allows us to conclude that
\[
\frac{C^2}{(F^*+M^*)\cdot C}\geq (2-2d)-\delta\frac{(F^*\cdot C)}{(F^*+ M^*)\cdot C}\geq 2-2d-\delta.\]\medskip

To finish our proof, we show Item (b). Set $\mathcal{C}=\{p_i\}_{i=1}^n$ and consider a rational first integral $F/G$ of $\mathcal{F}$ satisfying the conditions proved in Theorem \ref{thm_attch_fol}. Notice that $\conf\subseteq \Bf$. We start with the case $S_0=\PP^2$. The assumptions for $C$ imply that $C$ is not exceptional.

Then, $C$ is the strict transform of a curve in $\mathbb{P}^2$ and therefore it is linearly equivalent to the divisor ${b}L^*-\sum_{i=1}^n {m}_iE_i^*$, where ${b}=\deg((\pi_{\mathcal{C}})_*C)$ and, for each $i$, ${m}_i$ is the multiplicity of the strict transform of $(\pi_{\mathcal{C}})_*C$ at $p_i$. Moreover, since $C$ is not exceptional, $(\pi_{\mathcal{C}})_*C$ is a component of a member of the pencil of curves defined by the curves with equations $F=0$ and $G=0$ and, therefore, ${b}\leq d$.

Then,
\[
\frac{C^2}{L^*\cdot C}=\frac{{b}^2-\sum_{i=1}^n{m}_i^2}{{b}}\geq {b}-\sum_{i=1}^n {m}_i\geq {b}(1-n)\geq d(1-n).\]\medskip

Assume now that $S_0=\fd$. Here the curves $F$ and $G$ giving rise to the rational first integral $F/G$ of \fol have bidegree $({d_1},{d_2})$ such that ${d_1}\leq d$ and ${d_2}=d$.

$C$ is a reduced and irreducible $\tilde{\mathcal{F}}$-invariant curve on $S$ and it is linearly equivalent to the divisor ${a} F^*+{b} M^*-\sum_{i=1}^n {m}_iE_i^*$, where $({a},{b})$ is the bidegree of $(\pi_{\mathcal{C}})_*C$  and ${m}_i$ is defined as above. Reasoning as previously, ${a}\leq {d_1}$ and ${b}\leq {d_2}$. Moreover, by Lemma \ref{lem_bnd_deg}, $0\leq {m}_i\leq {a}+{b}+\delta {b}$. Thus, it holds one of the following three cases:
\begin{itemize}
\item  ${a}=1,\,{b}=0$ (that is, $(\pi_{\mathcal{C}})_*C$ is a fiber of $\mathbb{F}_\delta$). Then, \[
\frac{C^2}{(F^*+M^*)\cdot C}=-\sum_{i=1}^n{m}_i^2\geq -n.\]
\item  ${a}=-\delta,\,{b}=1$ (that is, $(\pi_{\mathcal{C}})_*C$ is linearly equivalent to the special section of $\mathbb{F}_\delta$). Then, \[
\frac{C^2}{(F^*+M^*)\cdot C}=-\delta-\sum_{i=1}^n{m}_i^2\geq-\delta-n.\]
\item ${a} \geq 0,\,{b}>0$. Then,
\begin{align*}
\frac{C^2}{(F^*+M^*)\cdot C}=&\frac{2{a} {b}+\delta {b}^2-\sum_{i=1}^n{m}_i^2}{{a} +{b}+\delta {b}}\geq\frac{2{a}{b}+\delta {b}^2-n({a}+{b}+\delta {b})^2}{{a}+{b}+\delta {b}}\\
\geq&-n({a}+{b}+\delta {b})\geq -n({d_1}+{d_2}+\delta {d_2})\geq -(\delta+2)dn.
\end{align*}
\end{itemize}
By considering the minimum of the right hand sides of the three obtained inequalities, we deduce the second inequality in Item (b) of the statement.

\end{proof}

\begin{remark}
Despite we use an auxiliary foliation \fol attached to a $S_0$-tuple $(S,S_0,\conf)$, the bounds for $C^2/(L^*\cdot C)$ (respectively, $C^2/((F^*+M^*)\cdot C)$) obtained in Theorem \ref{teo2} do not depend on \fol; they only depend on the values $n$ and $d$, which can be computed from the proximity graph of $\mathcal{C}$ described in Subsection \ref{subsec_conf}.

\end{remark}

We complete this subsection by giving explicit bounds related to the LWBNC for much wider families of divisors $D$ on any rational surface $S$.

We group our bounds in two corollaries whose proofs are supported in Theorem \ref{teo2} and run parallel to those of Corollaries \ref{a} and \ref{b}; therefore we omit them. For both results we fix an $S_0$-tuple $(S,S_0,\conf)$ as in Definition \ref{def_Stuple}.

Let us state the first one for which it is convenient to revise the definition of $\Delta(X;G,\epsilon)$ given in (\ref{eqn_Delta}).

\begin{corollary}\label{cor_expl1}
Let $\mathcal{F}$ be an attached to $(S,S_0,\mathcal{C})$ foliation and let $\tilde{\mathcal{F}}$ be its strict transform on $S$. Consider a positive real number $\epsilon$. Let $D$ be a nef divisor on $S$ such that
$D\in \Delta(S;L^*,\epsilon)$ (respectively, $D\in \Delta(S;F^*+M^*,\epsilon)$) if $S_0=\mathbb{P}^2$ (respectively, $S_0=\mathbb{F}_\delta$), and let $C$ be a non-exceptional negative curve on $S$ such that $D\cdot C>0$.

\begin{itemize}
\item[(a)] If $C$ is not $\tilde{\mathcal{F}}$-invariant, then
\begin{gather*}
\frac{C^2}{D\cdot C}\geq \frac{1}{\epsilon}(3-2d)\text{ if }S_0=\PP^2,\text{ and } \\
\frac{C^2}{D\cdot C}\geq \frac{1}{\epsilon}(2-2d-\delta)\text{ if }S_0=\fd,
\end{gather*}
where $d=\sum_{p\in O_\conf} d_{(\conf)_p}$, $d_{(\conf)_p}$ being the integer defined in (\ref{positiveinteger}).

\item[(b)] If $C$ is $\tilde{\mathcal{F}}$-invariant, then \begin{gather*}
\frac{C^2}{D\cdot C}\geq {\frac{1}{\epsilon}}d(1-n)\text{ if }S_0=\PP^2,\text{ and}\\
\frac{C^2}{D\cdot C}\geq {\frac{1}{\epsilon}}\min\{-n-\delta,-(\delta+2)dn\}\text{ if }S_0=\fd,
\end{gather*}
where $n=\#\conf$ and $d=\sum_{p\in O_\conf} d_{(\conf)_p}$, $d_{(\conf)_p}$ being the integer defined in (\ref{positiveinteger}).

\end{itemize}

As a consequence, the value $\nu_D(S)$ introduced in (\ref{eqn_nuD}) satisfies:
\begin{gather*}
\nu_{D}(S)\geq \min\left\{\frac{1}{\epsilon}(3-2d),{\frac{1}{\epsilon}}d(1-n),-\gamma \right\}\text{ if }S_0=\PP^2,\\
\nu_{D}(S)\geq \min\left\{\frac{1}{\epsilon}(2-2d-\delta),{\frac{1}{\epsilon}(}-n-\delta{)},{\frac{1}{\epsilon}(-\delta-2)}dn,-\gamma\right\}\text{ if }S_0=\fd,
\end{gather*}
where $n=\#\conf$, $\gamma$ is the maximum of the set $\{-E_q^2\mid q\in \mathcal{C}\}$ and $d=\sum_{p\in O_\conf} d_{(\conf)_p}$, $d_{(\conf)_p}$ being the integer defined in (\ref{positiveinteger}).

\end{corollary}

Our last corollary is the following one.

\begin{corollary}\label{cor_cotaejemplo}
Keep the same notation as in Corollary \ref{cor_expl1}. Let $D$ be a nef divisor on $S_0$ and let $C$ be a negative curve on $S$ such that $D^*\cdot C>0$.

\begin{itemize}

\item[(a)] If $C$ is not $\tilde{\mathcal{F}}$-invariant, then
\begin{gather*}
\frac{C^2}{D^*\cdot C}\geq 3-2d\text{ if }S_0=\PP^2,\text{ and }\\
\frac{C^2}{ D^* \cdot C}\geq 2-2d-\delta\text{ if }S_0=\fd,
\end{gather*}
where $d=\sum_{p\in O_\conf} d_{(\conf)_p}$, $d_{(\conf)_p}$ being the integer defined in (\ref{positiveinteger}).

\item[(b)] If $C$ is $\tilde{\mathcal{F}}$-invariant, then \begin{gather*}
\frac{C^2}{D^*\cdot C}\geq d(1-n)\text{ if }S_0=\PP^2,\text{ and}\\
\frac{C^2}{D^* \cdot C}\geq \min\{-n-\delta,-(\delta+2)dn\}\text{ if }S_0=\fd,
\end{gather*}
where $n=\#\conf$ and $d$ is defined as in Item (a).

\end{itemize}

As a consequence, the value $\nu_{D^*}(S)$ introduced in (\ref{eqn_nuD}) satisfies
\begin{gather*}
\nu_{{D^*}}(S)\geq \min\left\{3-2d,d(1-n) \right\}\text{ if }S_0=\PP^2,\\
\nu_{{D^*}}(S)\geq \min\left\{2-2d-\delta,-n-\delta,-(\delta+2)dn\right\}\text{ if }S_0=\fd.
\end{gather*}
\end{corollary}

\begin{remark}
	In \cite{CilFont}, when $S_0=\PP^2$, it is given a bound on $\frac{C^2}{L^*\cdot C}$ for all but finitely many reduced and irreducible curves on $S$.
\end{remark}

We finish the paper with an example of the computation of our bound on the value $\nu_{D^*}(S)$. We hope this makes it easier to read our results.

\begin{example}
Let $(S,S_0,\conf)$ an $S_0$-tuple, where $S_0=\PP^2$ (respectively, $\fd$). Assume that $\conf=\{p_i\}_{i=1}^{12}$ is a configuration over $S_0$ of cardinality $12$ whose proximity graph is that displayed in Figure \ref{fig_ex}.

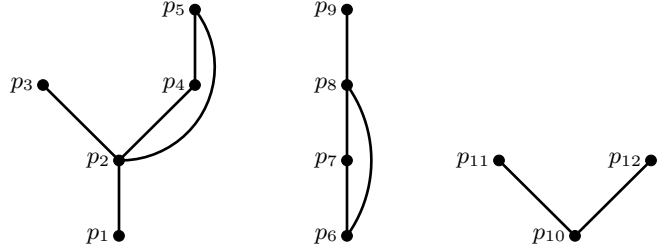
\begin{figure}[H]
\definecolor{ududff}{rgb}{0.30196078431372547,0.30196078431372547,1}
\definecolor{xdxdff}{rgb}{0.49019607843137253,0.49019607843137253,1}
\begin{tikzpicture}[line cap=round,line join=round,>=triangle 45,x=1cm,y=1cm]
\clip(-0.5,-0.5) rectangle (8.5,3.5);
\draw [line width=1pt] (1,0)-- (1,1);
\draw [line width=1pt] (1,1)-- (0,2);
\draw [line width=1pt] (1,1)-- (2,2);
\draw [line width=1pt] (2,2)-- (2,3);
\draw [line width=1pt] (4,3)-- (4,0);
\draw [line width=1pt] (6,1)-- (7,0);
\draw [line width=1pt] (8,1)-- (7,0);
\draw [shift={(2.6,1)},line width=1pt]  plot[domain=-0.6202494859828214:0.6202494859828215,variable=\t]({1*1.7204650534085253*cos(\t r)+0*1.7204650534085253*sin(\t r)},{0*1.7204650534085253*cos(\t r)+1*1.7204650534085253*sin(\t r)});
\draw [shift={(1.02,2.24)},line width=1pt]  plot[domain=-1.5869239606385328:0.6596287426369204,variable=\t]({1*1.2401612798341997*cos(\t r)+0*1.2401612798341997*sin(\t r)},{0*1.2401612798341997*cos(\t r)+1*1.2401612798341997*sin(\t r)});
\begin{scriptsize}
\draw [fill] (1,0) circle (2pt) node[anchor=east]{$p_1$};
\draw [fill] (1,1) circle (2pt) node[anchor=east]{$p_2$};
\draw [fill] (0,2) circle (2pt) node[anchor=east]{$p_3$};
\draw [fill] (2,2) circle (2pt) node[anchor=east]{$p_4$};
\draw [fill] (2,3) circle (2pt) node[anchor=east]{$p_5$};
\draw [fill] (4,0) circle (2pt) node[anchor=east]{$p_6$};
\draw [fill] (4,1) circle (2pt) node[anchor=east]{$p_7$};
\draw [fill] (4,2) circle (2pt) node[anchor=east]{$p_8$};
\draw [fill] (4,3) circle (2pt) node[anchor=east]{$p_9$};
\draw [fill] (7,0) circle (2pt) node[anchor=east]{$p_{10}$};
\draw [fill] (6,1) circle (2pt) node[anchor=east]{$p_{11}$};
\draw [fill] (8,1) circle (2pt) node[anchor=east]{$p_{12}$};
\end{scriptsize}
\end{tikzpicture}
\caption{Proximity graph of $\conf$.}\label{fig_ex}
\end{figure}
With the notation as introduced before Lemma \ref{lemanuevo}, we set $O_{\conf}=\{p_1,p_6,p_{10}\}$ and
\begin{gather*}
\widehat{(\conf)}_{p_1}=\{p_1,p_2,p_3,q_1,p_4,p_5\},\\
\widehat{(\conf)}_{p_6}=\{p_6,p_7,p_8,p_9,q_2\}\text{, and}\\
\widehat{(\conf)}_{p_{10}}=\{p_{10},p_{11},q_3,p_{12},q_4\},
\end{gather*}
 where $q_1$ (respectively, $q_2$, $q_3$ and $q_4$) is the only satellite point in the exceptional divisor given by blowing-up at $p_3$ (respectively, $p_9$, $p_{11}$ and $p_{12}$).

 The values defined in (\ref{positiveinteger}) are $d_{(\conf)_{p_1}}=10$, $d_{(\conf)_{p_6}}=7$ and $d_{(\conf)_{p_{10}}}=6$ and by Corollary \ref{cor_cotaejemplo}, one deduces that, for any nef divisor $D$ on $S_0$,\[
\nu_{D^*}(S)\geq\left\{
\begin{array}{ll}
\min\left\{3-2d,d(1-n)\right\}=\min\left\{-43,-345\right\}=-345&\text{if }S_0=\PP^2,\\
&\\
\begin{array}{l}
	\hspace{-0.42em}
	\min\left\{2-2d-\delta,-n-\delta,-(\delta+2)dn\right\}\\
	\hspace{-0.42em}=\min\left\{-44 - \delta, -16 - \delta, -368  (2 + \delta)\right\}=-368  (2 + \delta))
\end{array} &\text{if }S_0=\fd.
\end{array}\right.\]
\end{example}



\bibliographystyle{plain}
\bibliography{MIBIBLIO}
\end{document}